\documentclass[11pt,psamsfonts,leqno,oneside,letterpaper,left=1.25truein, top=1truein, margin=2.5cm ]{article}
\usepackage[letterpaper,left=1.75truein, top=1truein]{geometry}
\usepackage{amssymb,amsmath, amsthm, amssymb, amscd,graphics}
\usepackage{epsfig}
\usepackage{color}

\usepackage[colorlinks,linkcolor=blue,citecolor=blue]{hyperref}
\usepackage{indentfirst}
\usepackage{geometry}
\usepackage{diagrams}
\diagramstyle[labelstyle=\scriptstyle]

\newtheorem{theorem}{Theorem}[section]
\newtheorem{lemma}[theorem]{Lemma}
\newtheorem{corollary}[theorem]{Corollary}
\newtheorem{conjecture}[theorem]{Conjecture}
\newtheorem{definition}[theorem]{Definition}
\newtheorem{proposition}[theorem]{Proposition}
\newtheorem{remark}[theorem]{Remark}

\newtheorem{problem}[theorem]{Problem}

\newtheorem{scholium}[theorem]{Scholium}

\def\PSL{\mbox{\rm{PSL}}}

\def\Isom{\mbox{\rm{Isom}}}

\def\orb{{\rm{orb}}}
\def\omin{\ensuremath{\mathcal{O}_{\text{min}}}(K)}

\def\gcd{\hbox{gcd}}

\newcommand{\HF}{\widehat{HF}}
\newcommand{\HFK}{\widehat{HFK}}
\newcommand{\s}{\mathfrak{s}}

\newcommand{\x}{{\bf x}}
\newcommand{\y}{{\bf y}}
\newcommand{\Spinc}{\text{Spin}^c}
\newcommand{\RSpinc}{\underline {\text{Spin}^c}}
\newcommand{\alphab}{\mbox{\boldmath$\alpha$}}
\newcommand{\betab}{\mbox{\boldmath$\beta$}}
\newcommand{\zz}{\mathbb{Z}}

\title{Knot commensurability and the Berge conjecture}

\author{M. Boileau\thanks{Partially supported
by Institut Universitaire de France.}, S. Boyer\thanks{Partially supported by NSERC grant OGP0009446.}, R. Cebanu\thanks{Partially supported from the contract PN-II-ID-PCE 1188 265/2009.} \ \& G. S.  Walsh\thanks{Partially supported by N. S. F. grant 0805908.}}

\date{\today}
\begin{document}

\maketitle

\abstract{{\scriptsize We investigate commensurability classes of hyperbolic knot complements in the generic case of knots without hidden symmetries. We show that such knot complements which are commensurable are cyclically commensurable, and that there are at most $3$ hyperbolic knot complements in  a cyclic commensurability class. Moreover if two  hyperbolic knots have cyclically commensurable complements, then they are fibered with the same genus and are chiral. A characterisation of cyclic commensurability classes of complements of periodic knots is also given. In the non-periodic case, we reduce the characterisation of cyclic commensurability classes to a generalization of the Berge conjecture.}}

\section{Introduction}

We work in the oriented category throughout this paper. In particular we endow the complement of any knot $K \subset S^3$ with the orientation inherited from the standard orientation on $S^3$. We consider two knots to be equivalent if there is an orientation-preserving homeomorphism of $S^3$ taking one to the other. Covering maps will be assumed to preserve orientation unless stated otherwise.  

Two oriented orbifolds are \emph{commensurable} if they have homeomorphic finite sheeted covers.  We are interested in studying commensurability classes of knot complements in $S^3$. By abuse of language we will say that two knots in the $3$-sphere are commensurable if their complements are commensurable. Set $$\mathcal{C}(K) = \{ \hbox{knots } K' \subset S^3 : K' \hbox{ is commensurable with }  K\}.$$  
A difficult and widely open problem is to describe commensurability classes of knots. 

One of our main concerns is to provide a priori bounds on the number of hyperbolic knots in a given commensurability class. Unless otherwise stated, knots are considered to be in $S^3$. Hence in this paper
$K \subset S^3$ will be a hyperbolic knot. Its complement $S^3 \setminus K ={\bf H}^3/\Gamma_K$  is a complete, oriented, hyperbolic $3$-manifold of finite volume, where 
$\pi_1(S^3 \setminus K ) \cong \Gamma_K \subset \PSL(2, \mathbb{C}) = \Isom^+({\bf H}^3)$ is a lattice. Any knot $K'$ commensurable with $K$ is also hyperbolic and the commensurability condition is equivalent to the fact that $\Gamma_K$ and some conjugate of $\Gamma_{K'}$ in $\hbox{Isom}^+({\bf H}^3)$ have a common finite index subgroup. 

Recall that the {\it commensurator} of  a group $\Gamma$ in $\PSL(2, \mathbb{C})$ is 
$$C^+(\Gamma) = \lbrace g \in \PSL(2, \mathbb{C}): [\Gamma: \Gamma \cap  g^{-1}\Gamma g] < \infty  
\hbox{ and } [g^{-1} \Gamma g: \Gamma \cap g^{-1}\Gamma g] < \infty \rbrace.$$ 
Then  $K$ and $K'$ are commensurable if and only if $C^+(\Gamma_K)$ and $C^+(\Gamma_{K'})$ are conjugate in $\PSL(2, \mathbb{C})$. An element $g \in  C^+(\Gamma_K)$ induces an orientation-preserving isometry  between two finite sheeted coverings  of  $S^3 \setminus K$. It is called a {\it hidden symmetry} of $K$ if it is not the lift of an isometry of  $S^3 \setminus K$.

The group of isotopy classes of orientation-preserving symmetries of $(S^3,K)$ is isomorphic, in the obvious way, to $\Isom^+(S^3 \setminus K)$, the group of orientation-preserving isometries of $S^3 \setminus K$. It is also isomorphic to the quotient group $N^+(K)/\Gamma_K$  where $N^+(K)$ is the normalizer of $\Gamma_K$ in $\PSL(2, \mathbb{C})$. We will use either description as convenient.  Then $K$ has  hidden symmetries if and only if  $N^+(K)$ is strictly smaller than  $C^+(\Gamma_K)$.  Hyperbolic knots with hidden symmetries appear to be rare, as Neumann and Reid \cite{NR} showed that if $K$ has hidden symmetries then the cusp shape of ${\bf H}^3/\Gamma_K$ is contained in $\mathbb{Q}[i]$ or $\mathbb{Q}[\sqrt{-3}]$. 

Currently, the only knots known to admit hidden symmetries are the figure-8 and the two dodecahedral knots of Aitchison and Rubinstein described in \cite{AR} (c.f. Conjecture \ref{conj:hidden} below). These three knots have cusp field $\mathbb{Q}[\sqrt{-3}]$.  There is one known example of a knot with cusp field $\mathbb{Q}[i]$, and it does not admit hidden symmetries. See Boyd's notes \cite[page 17]{Boyd} and Goodman, Heard and Hodgson \cite{GHH}. 

It is a fundamental result of Margulis that a finite co-volume, discrete subgroup $\Gamma$ of $\PSL(2, \mathbb{C})$ is non-arithmetic if and only if  there is a unique minimal orbifold in the commensurability class of ${\bf H}^3/\Gamma$, namely ${\bf H}^3/C^+(\Gamma)$. Reid \cite{Re1} has shown that the figure-8 is the only arithmetic knot (i.e. knot with arithmetic complement) in $S^3$, hence it is the unique knot in its commensurability class. So in what follows we only consider non-arithmetic knots. In particular, $C^+(\Gamma_K)$ is a lattice in $\PSL(2, \mathbb{C})$ and  the unique minimal element in the commensurability class of $S^3 \setminus K = {\bf H}^3/\Gamma_K$ is the oriented orbifold ${\bf H}^3/C^+(\Gamma_K)$, which we denote by $\omin$. 

When $K$ has no hidden symmetries,  
$$\omin = {\bf H}^3/N^+(K) = (S^3 \setminus K)/\Isom^+(S^3 \setminus K).$$The positive solution of the Smith conjecture implies that $\Isom^+(S^3 \setminus K)$ is cyclic or dihedral and the subgroup of $\Isom^+(S^3 \setminus K)$ which acts freely on $K$ is cyclic of index at most $2$. We denote this subgroup by $Z(K)$. Clearly the oriented orbifold 
$$\mathcal{Z}_{K} = (S^3 \setminus K)/Z(K)$$ 
has a torus cusp and either coincides with the minimal element in the commensurability class of $S^3 \setminus K$ or is a $2$-fold cover of it. Hence in this case the cusp of $\omin$ is {\it flexible}: its horospherical cross-section  is either $T^2$ or $S^{2}(2,2,2,2)$.  
Neumann and Reid \cite{NR} proved that a non-arithmetic knot $K$ has no hidden symmetries if and only if $\omin$ has a flexible cusp and further, that this condition is equivalent to the fact that 
$S^3 \setminus K$ normally covers $\omin$.  If a commensurability class has a unique minimal element with a single cusp and the cusp is flexible, we call the commensurability class itself {\it flexible}. When $K$ does admit hidden symmetries, the horospherical cross-section of $\omin$ is a Euclidean turnover, which is rigid. If a commensurability class has a unique minimal element with a single cusp which is rigid, we say that the commensurability class itself is \emph{rigid}. 

Reid and Walsh \cite{RW} proved that a hyperbolic 2-bridge knot is unique in its commensurability class and raised the following conjecture:
 
 \begin{conjecture} \label{RWconjecture} {\rm (Reid-Walsh \cite{RW})}
For a hyperbolic knot $K \subset S^3$,  $|\mathcal{C}(K)| \leq 3$.
\end{conjecture} 

See also \cite[Theorem 2]{Neu}. 

The commensurability class of the $(-2,3,7)$ pretzel knot is flexible \cite{MatMac} and contains exactly three knots. 
Neil Hoffman \cite{Hof} has constructed an infinite family of hyperbolic knots with this property.

Our first result proves the  conjecture in the generic case:   

\begin{theorem} \label{th:conjecture} 
A flexible commensurability class contains at most three hyperbolic knot complements.
\end{theorem}

\noindent A precise formulation of the expected genericity of the flexible case is contained in the following conjecture of Neumann and Reid: 

\begin{conjecture} \label{conj:hidden} {\rm (Neumann-Reid)}
The only rigid commensurability class containing hyperbolic knot complements is the commensurability class of the dodecahedral knots, and there are only two knot complements in this class. 
\end{conjecture} 

We say that two hyperbolic orbifolds are \emph{cyclically commensurable} if they have a common finite cyclic cover. We denote by $\mathcal{CC}(K)$ the set of hyperbolic knots cyclically commensurable with $K$. A priori cyclic commensurability is much more restrictive than commensurability. However for hyperbolic knots without hidden symmetries, the commensurability class and the cyclic commensurability class coincide. Theorem \ref{th:conjecture} follows immediately from the following results:

\begin{theorem} \label{thm:cyclic} 

$(1)$ Knots without hidden symmetries which are commensurable are cyclically commensurable.\\
$(2)$ A cyclic commensurability class contains at most three hyperbolic knot complements.
\end{theorem} 

In this article we analyze the case of hyperbolic knots which are commensurable to other hyperbolic knots and which do not admit hidden symmetries. However, many of our results hold for any hyperbolic knots with hidden symmetries which are cyclically commensurable to other knots. This conjecturally does not happen (see Conjecture \ref{conj:rigid}).  

Geometrisation combines with the work of Gonz{\'a}lez-Acu{\~n}a and Whitten \cite{GAW} to  determine close connections between the family of knots which are cyclically commensurable to other knots and the family of knots which admit lens space surgeries: if the complement of a knot $K$ is covered by another knot complement, then the covering is cyclic and this occurs if and only if $K$ admits a non-trivial lens space surgery. In this situation, a fundamental result of Ni \cite{NiF} implies that $K$ is fibred. Here we show that distinct knots without hidden symmetries which are  commensurable are obtained from primitive knots in orbi-lens spaces (\S \ref{orbi-lens}) which admit non-trivial orbi-lens space surgeries. Further, we prove an analogue of Ni's result in the orbifold setting: 

\begin{theorem} \label{thm: Ni analogue} 
Let $K$ be a knot in an orbi-lens space $\mathcal{L}$ which is primitive in $\mathcal{L}$. If $K$ admits a non-trivial orbi-lens space surgery, then the exterior of $K$ admits a fibring by $2$-orbifolds with base the circle.  
\end{theorem}

\noindent Our next result is an interesting by-product of the method of proof of Theorem \ref{thm: Ni analogue}. For the definition of a $1$-bridge braid in a solid torus we refer to \S \ref{unwrapped}.

\begin{proposition} \label{fibred faces} 
Let $M$ be the exterior of a hyperbolic $1$-bridge braid in a solid torus $V$. Then each top-dimensional face of the Thurston norm ball in $H_2(M, \partial M; \mathbb R)$ is a fibred face.
\end{proposition}


\begin{theorem}\label{th:properties} 
Let $K$ be a hyperbolic knot. If $|\mathcal{CC}(K)| \geq 2$ then:\\
$(1)$ $K$ is a fibred knot.\\
$(2)$ the genus of $K$ is the same as that of any $K' \in \mathcal{CC}(K)$.\\ 
$(3)$ the volume of $K$ is different from that of any $K' \in \mathcal{CC}(K) \setminus {K}$.  In particular, the only mutant of $K$ contained in $\mathcal{CC}(K)$ is $K$. \\
$(4)$ $K$ is chiral and not commensurable with its mirror image. 
\end{theorem}

In particular this result holds for a hyperbolic knot $K$ without hidden symmetries and any $K' \in \mathcal{C}(K) \setminus K$.

We pause to note the marked difference between the case of flexible and rigid commensurability classes containing knot complements. Recall that the commensurability class of the two dodecahedral knots \cite{AR} is the only known rigid commensurability class containing knot complements. These knots do not satisfy any of the conditions above: one dodecahedral knot is fibred, the other isn't; the knots have different genus; they have the same volume; the knots are both amphichiral \cite[12.1]{AR2}.  In addition, they are not cyclically commensurable in contrast with Theorem \ref{thm:cyclic}. 

 A knot $K$ is {\it periodic} if it admits a non-free symmetry with an axis disjoint from $K$. As a consequence of the works of Berge \cite{Bergetori} and Gabai \cite{Gabai1} we obtain the following  characterisation of cyclic commensurability classes of periodic knots. 
 We refer to \S \ref{unwrapped} for the definitions of Berge-Gabai knots  and unwrapped Berge-Gabai knots.

 \begin{theorem}\label{th:periodic} 
Let $K$ be a periodic hyperbolic knot.
 If $|\mathcal{CC}(K)| \geq 2$ then:\\
 $(1)$ $K$ has a unique axis of symmetry disjoint from $K$.\\
 $(2)$ $K$ is obtained by unwrapping a Berge-Gabai knot $\bar K$ in an orbi-lens space. In particular $K$ is strongly invertible.\\
 $(3)$ each $K' \in \mathcal{CC}(K)$ is determined by unwrapping the Berge-Gabai knot represented by the core of the surgery solid torus in an orbi-lens space obtained by Dehn surgery along $\bar K$.
 \end{theorem}

In particular this result holds for a periodic hyperbolic knot $K$ without hidden symmetries and any $K' \in \mathcal{C}(K)$.

The proof of Theorem \ref{th:periodic} reduces the characterisation of hyperbolic knots $K \subset S^3$ such that $|\mathcal{CC}(K)| \geq 2$  to the case where $Z(K)$ acts freely on $S^3$ and to the construction of all primitive knots in a lens space with a non-trivial lens space surgery.  It is a result of Bonahon and Otal \cite{Bonotal} that for each $g \geq 1$, a lens space admits a unique genus $g$ Heegaard splitting, which is a stabilization of the genus 1 splitting. 

 \begin{problem}\label{berge} 
Characterize primitive knots $\bar K$ in a lens space $L$ which admit a non-trivial lens space surgery. In particular, is every such knot a doubly primitive knot on the genus 2 Heegaard surface of  $L$?
\end{problem}

Suppose $\bar K$ is in $S^3$.  Then this problem is the setting of the Berge conjecture, which contends that a knot in $S^3$ which admits a non-trivial lens space surgery is doubly primitive. Doubly primitive knots are knots which lie on the genus 2 Heegaard surface in such a way that the knot represents a generator of the fundamental group of each handlebody.

A fundamental result of Schwartz \cite{Sch} implies that the fundamental groups of two hyperbolic knots $K, K'$ are quasi-isometric if and only if $K'$ is commensurable with $K$ or with its mirror image. Proposition \ref{prop:volume} below shows that a knot without hidden symmetries cannot be commensurable to its mirror image. Therefore, as a consequence of Theorems \ref{thm:cyclic}, \ref{th:properties} and \ref{th:periodic} we obtain the following results on quasi-isometry classes of knot groups:

\begin{corollary}\label{cor:uniqueness} Let $K$ be a  hyperbolic knot without hidden symmetries. Then there are at most three knots $K'$ with group $\pi_1(S^3 \setminus K ')$ quasi-isometric to $\pi_1(S^3 \setminus K )$. Moreover $\pi_1(S^3 \setminus K )$ is the unique knot group in its quasi-isometry class in the following cases:\\
$(i)$ $K$ is not fibred.\\
$(ii)$ $K$ is amphichiral.\\
$(iii)$ $K$ is periodic and is not an unwrapped Berge-Gabai knot; for instance, $K$ is periodic but not strongly invertible.\\
$(iv)$ $K$ is periodic with two distinct axes of symmetry. 
\qed
\end{corollary}

The paper is organized as follows.  Theorem \ref{thm:cyclic} is proved in \S \ref{3conj}. Theorem \ref{th:periodic} and (3) of Theorem \ref{th:properties} are contained in \S \ref{unwrapped}. Theorem \ref{thm: Ni analogue}, parts (1) and (2) of Theorem \ref{th:properties}, and Proposition \ref{fibred faces} are proven in \S \ref{fibred}. Part (4) of Theorem \ref{th:properties} is proven in \S \ref{amphi}. Sections \ref{conventions} and \ref{orbi-lens} are devoted to conventions and background on certain spherical orbifolds. \\

\noindent {\bf Acknowledgements}. We thank Jacob Rasmussen for explaining the proof that a knot in a lens space which admits a non-trivial lens space surgery has fibred complement.  We also benefited from helpful conversations with Walter Neumann and Alan Reid.

\section{Slopes, Dehn filling and cusp types} \label{conventions}

A {\it slope} on a torus $T$ is an isotopy class of essential simple closed curves. The set of slopes on $T$ corresponds bijectively, in the obvious way, with $\pm$-classes of primitive elements of $H_1(T)$. Thus to each slope $r$ we associate the primitive classes $\pm \alpha(r) \in H_1(T)$ represented by a simple closed curve in the class of $r$. The {\it distance} between two slopes $r, r'$ on $T$ is given by $\Delta(r, r') = |\alpha(r) \cdot \alpha(r')|$. 

Given a slope $r$ on on a torus boundary component $T$ of a $3$-manifold $M$, the {\it $r$-Dehn filling} of $M$ with slope $r$ is the $3$-manifold $M(T; r) := (S^1 \times D^2) \cup_f M$ where $f$ is any homeomorphism $\partial(S^1 \times D^2) \to T$ such that $f(\{*\} \times \partial D^2)$ represents $r$. It is well-known that $M(T; r)$ is independent of the choice of $f$. When there is no risk of ambiguity, we shall usually denote $M(T;r)$ by $M(r)$. 

Recall that topologically, a cusp of a complete, finite volume, orientable, hyperbolic $3$-orbifold is of the form $\mathcal{B} \times \mathbb R$ where $\mathcal{B}$ is a closed, connected, orientable, Euclidean $2$-orbifold. In this case, we say that the cusp is a $\mathcal{B}$ {\it cusp}. 

A {\it slope} $r$ in a torus cusp of a complete, non-compact, finite volume hyperbolic $3$-orbifold $\mathcal{O}$ is a cusp isotopy class of essential simple closed curves which lie on some torus section of the cusp. Inclusion induces a bijection between the slopes on a torus cross-section of the cusp with those in the cusp, and we identify these sets below. 

\begin{lemma} \label{lem:slopes well-defined} 
Let $\mathcal{O}$ be a complete, finite volume, orientable, hyperbolic $3$-orbifold which has one end, a torus cusp, and let $r$ be a slope in this cusp. Then for any orientation-preserving homeomorphism $f: \mathcal{O} \to \mathcal{O}$, the slope $f(r)$ equals $r$. 
\end{lemma} 

\begin{proof}
Our assumptions imply that $|\mathcal{O}|$ is the interior of a compact, connected, orientable $3$-manifold $M$ with torus boundary to which we can extend $f$. To prove the lemma it suffices to show that $f$ acts as multiplication by $\pm 1$ on $H_1(\partial M)$. 
First note that $f_*(\lambda_M) = \epsilon \lambda_M$ where $\epsilon \in \{\pm 1\}$ and $\lambda_M \in H_1(\partial M)$ is the rational longitude of $M$. (Thus $\pm \lambda_M$ are the only primitive classes in $H_1(\partial M) \equiv \pi_1(\partial M)$ which are trivial in $H_1(M; \mathbb Q)$.) 

Let $\rho: \pi_1(\mathcal{O}) \to \PSL(2, \mathbb{C})$ be a discrete faithful representation. By Mostow-Prasad rigidity, there is an element $A \in \PSL(2, \mathbb{C})$ such that $\rho \circ f_\# = A \rho A^{-1}$. In particular, 
$$\rho(\lambda_M)^\epsilon = \rho(f_\#(\lambda_M)) = A \rho(\lambda_M) A^{-1} .$$ 
Without loss of generality we can assume that $\rho(\lambda_M)$ is upper triangular and parabolic. Then it is easy to verify that $A$ is upper triangular and parabolic when $\epsilon = 1$ or upper triangular with diagonal entries $\pm i$ when $\epsilon = -1$. A simple calculation then shows that when $\epsilon = 1$, $\rho(f_\#(\gamma)) = \rho(\gamma)$ for each $\gamma \in \pi_1(\partial M)$, which implies that $f_*$ is the identity. Similarly when $\epsilon = -1$ it's easy to see that $f_* = -I$. 
\end{proof}

Given two slopes $r, r'$ in the cusp, the reader will verify that the distance between two of their representatives contained in some torus cross-section of the cusp is independent of the cross-section, and we define the {\it distance} between $r$ and  $r'$, denoted $\Delta(r,r')$, to be this number. 

Let $r$ be a slope in a torus cusp of $\mathcal{O}$ and $\mathcal{\hat O}$ an orbifold obtained by truncating $\mathcal{O}$ along the cusp. The {\it Dehn filling of $\mathcal{O}$ of slope $r$}, denoted $\mathcal{O}(r)$, is the $r$-Dehn filling of $\mathcal{\hat O}$.

\section{Orbi-lens spaces}\label{orbi-lens}

We denote the singular set of an orbifold $\mathcal{O}$ by $\Sigma(\mathcal{O})$ throughout the paper. 

An {\it orbi-lens space} is the quotient orbifold of $S^3$ by a finite cyclic subgroup of $SO(4)$. We begin by examining their structure. 

The first homology group of an orbifold is the abelianisation of its fundamental group. 

A knot in an orbi-lens space $\mathcal{L}$ is {\it primitive} if it carries a generator of $H_1(\mathcal{L})$. 

\begin{lemma} \label{s3cyclicquotient} 
Let $Z$ be a finite cyclic subgroup of $SO(4)$ of order $n$ and fix a generator $\psi$ of $Z$. There are a genus one Heegaard splitting $S^3 = V_1 \cup V_2$, cores $C_1, C_2$ of $V_1, V_2$, and integers $a_1, a_2 \geq 1$ such that \\
$(1)$ both $V_1$ and $V_2$ are $Z$-invariant. \\
$(2)$ $\psi$ acts by rotation of order $a_1$ on $C_1$ and order $a_2$ on $C_2$. Moreover, the $Z$-isotropy subgroup of a point in 
\vspace{-.15cm} 
\begin{itemize}
\vspace{-.15cm} \item[\hspace{3mm} $\bullet$ ] $S^3 \setminus (C_1 \cup C_2)$ is trivial. 
\vspace{-.05cm} \item[\hspace{3mm} $\bullet$ ] $C_1$ is generated by $\psi^{a_1}$ and has order $\bar a_2 = n/a_1$, 
\vspace{-.05cm} \item[\hspace{3mm} $\bullet$ ] $C_2$ is generated by $\psi^{a_2}$ and has order $\bar a_1 = n/a_2$.  
\end{itemize}
\vspace{-.1cm} Thus $n = \hbox{lcm}(a_1, a_2), \bar a_1 = a_1/\gcd(a_1,a_2), \bar a_2 = a_2/\gcd(a_1,a_2)$, so $\gcd(\bar a_1, \bar a_2) = 1$. \\
$(3)$ $|S^3/Z|$ is the lens space with fundamental group $\mathbb Z / \gcd(a_1, a_2)$ and genus one Heegaard splitting $(V_1/Z) \cup (V_2/Z)$. The ramification index of a point $y \in |S^3/Z|$ is $\bar a_2$ if $y \in C_1/Z$, $\bar a_1$ if  $y \in C_2/Z$, and $1$ otherwise. Hence $\Sigma(S^3/Z) \subseteq (C_1/Z) \cup (C_2/Z)$. 
\end{lemma}

\begin{proof} We can find two mutually orthogonal $2$-dimensional subspaces of $\mathbb{R}^4$ on which $\psi$ acts by rotation. Thus if we think of these subspaces as the two coordinate planes of $\mathbb C^2$,  $\psi$ has the form 
$$\psi(z, w) = (e^{{2\pi i \alpha_1}/a_1}z, e^{{2\pi i \alpha_2}/a_2}w)$$ 
where $\gcd(\alpha_1, a_1) = \gcd(\alpha_2, a_2) = 1$ and $n =
\hbox{lcm}(a_1,a_2)$. The subgroup of $Z$ which 
\vspace{-.1cm} 
\begin{itemize}
\item[\hspace{3mm} $\bullet$] fixes $(z,w)$ with $zw \ne 0$ is the trivial subgroup. 
\item[\hspace{3mm} $\bullet$] fixes the $z$-plane is generated by $\psi^{a_1}$ and has order $\bar a_2 = n/a_1 = a_2/\gcd(a_1,a_2).$ 
\item[\hspace{3mm} $\bullet$] fixes the $w$-plane is generated by $\psi^{a_2}$ and has order $\bar a_1 = n/a_2 = a_1/\gcd(a_1,a_2).$ 
\end{itemize}
\vspace{-.1cm}
The genus one Heegaard splitting of $S^3$ given by the two solid tori $V_1 = \{(z, w) : |z|^2 + |w|^2 = 1, |w| \leq 1/\sqrt{2}\}$ and $V_2 = \{(z, w) : |z|^2 + |w|^2 = 1, |z| \leq 1/\sqrt{2}\}$ is invariant under $Z$ and determines a genus one Heegaard splitting of $|S^3 / Z(K)|$. Further, the isotropy subgroup of a point $(z, w) \in S^3$ is trivial if $|zw| \ne 0$, $\mathbb Z / \bar a_2$ if $w = 0$, and $\mathbb Z / \bar a_1$ if $z = 0$. The conclusions of the lemma follow from these observations. 
\end{proof} 

\begin{corollary} \label{equivorbilens} 
A $3$-orbifold $\mathcal{L}$ is an orbi-lens space if and only if $|\mathcal{L}|$ is a lens space which admits a genus one Heegaard splitting $|\mathcal{L}| = V_1 \cup V_2$ such that $\Sigma(\mathcal{L})$ is a closed submanifold of the union of the cores $C_1, C_2$ of $V_1, V_2$, and there are coprime positive integers $b_1, b_2 \geq 1$ such that a point of $C_j$ has isotropy group $\mathbb Z / b_j$. In the latter case, $\pi_1(\mathcal{L}) \cong \mathbb Z / (b_1 b_2 |\pi_1(|\mathcal{L}|)|)$. 
\end{corollary}

\begin{proof} Lemma \ref{s3cyclicquotient} shows that an orbi-lens space has the form claimed in the corollary. Conversely, suppose that $\mathcal{L}$ is a $3$-orbifold for which $|\mathcal{L}| \cong L(p,q)$ admits a genus one Heegaard splitting $|\mathcal{L}| = V_1 \cup V_2$ such that $\Sigma(\mathcal{L})$ is a closed submanifold of the union of the cores $C_1, C_2$ of $V_1, V_2$, and there are coprime positive integers $b_1, b_2 \geq 1$ such that a point of $C_j$ has isotropy group $\mathbb Z / b_j$. It is straightforward to verify that there is a $\mathbb Z / b_1 b_2 p$-fold cyclic cover $S^3 \to \mathcal{L}$ whose deck transformations lie in $SO(4)$. Thus  $\mathcal{L}$ is an orbi-lens space. 
\end{proof}

We will use $\mathcal{L}(p,q; b_1, b_2)$ to denote the orbifold described in the corollary. As we are mainly concerned with the case $b_1 = 1$ and $b_2 = a$, we use $\mathcal{L}(p,q; a)$ to denote $\mathcal{L}(p,q; 1, a)$. When $a = 1$, $\mathcal{L}(p,q; a)$ is just $L(p,q)$.

\section{Proof of Theorem \ref{thm:cyclic}}  \label{3conj} 

We start by proving: 

\begin{proposition} \label{prop:cyclicdef}
Two hyperbolic knot complements have a finite-index cyclic cover if
and only if they have a finite-index cyclic quotient. Moreover, two
cyclically commensurable hyperbolic knot complements have the same
normalizers in $\PSL(2, \mathbb C)$. \end{proposition}

\begin{proof}  The fact that a common finite-index cyclic quotient
implies a common finite-index cyclic cover is immediate from the
isomorphism theorems. Consider the converse then.

Suppose that $S^3 \setminus K_1 \cong {\bf H}^3/ \Gamma_1$ and $S^3
\setminus K_2 \cong {\bf H}^3 / \Gamma_2$ have a common finite-index
cyclic cover $M \cong {\bf H}^3/\Gamma_M$.  We may assume, after
conjugating, that $\Gamma_M \subseteq \Gamma_1 \cap \Gamma_2$.
Since knot complements have unique cyclic covers of a given order,
each isometry of $S^3 \setminus K_j$ is covered by an isometry of $M$
for $j = 1, 2$. Recall that the cyclic subgroup $Z(K) \subset \Isom^+(S^3 \setminus K)$ acts freely on $K$ and that the orbifold $\mathcal{Z}_{K}  = (S^3 \setminus K)/Z(K)$ has a torus cusp. 
 Let $\widetilde{Z(K_1)}$ be the subgroup of
$\hbox{Isom}^+(M)$ covering $Z(K_1)$ and define $\widetilde{Z(K_2)}$
similarly. By construction, $\widetilde{Z(K_1)}$ and
$\widetilde{Z(K_2)}$ act freely in the cusp of $M$ and since this cusp
is unique, the lift of each element of $\widetilde{Z(K_j)}$ to ${\bf
H}^3$ is parabolic ($j = 1, 2$). Thus $\widetilde{Z(K_1)}$ and
$\widetilde{Z(K_2)}$ act by Euclidean translations in each horotorus
of the cusp of $M$. It follows that $\widetilde{Z(K_1)}$ and
$\widetilde{Z(K_2)}$ generate an abelian subgroup $A$ of
$\hbox{Isom}^+(M)$ which acts by Euclidean translations in the horotori.
Thus there are regular covers $S^3 \setminus K_j \to \mathcal{O} =
M/A$ for $j = 1, 2$ where $\mathcal{O}$ has a torus cusp. These covers
are cyclic by \cite{GAW} and \cite[Lemma 4]{Re1} and so $S^3 \setminus
K_1$ and $S^3 \setminus K_2$ have a common cyclic quotient. Since the
covering group of $M \rightarrow S^3 \setminus K_i$ descends to
$Z(K_j)$ by the same argument, $M/A = \mathcal{Z}_{K_1} =
\mathcal{Z}_{K_2}$. Note, moreover, that there is a cover $S^3
\setminus K_1 \to (S^3 \setminus K_2)/\hbox{Isom}^+(S^3 \setminus K_2)
= {\bf H}^3/N^+(\Gamma_2)$ which is regular by \cite{GAW} and
\cite[Lemma 4]{Re1}. Thus $N^+(\Gamma_2) \subseteq N^+(\Gamma_1)$.
Similarly $N^+(\Gamma_1) \subseteq N^+(\Gamma_2)$, so these
normalizers are equal. This completes the proof.
\end{proof}

An immediate corollary is a strengthened  version of \cite [Thm2.2]{CaDu} for hyperbolic knots in $S^3$.

\begin{corollary} Two hyperbolic knot complements are cyclically commensurable if and only if they have a common regular finite cover with a single cusp.
\end{corollary}

\begin{proof}
The forward implication is obvious, as a finite cyclic cover of a knot complement has one cusp. Let $K_1$ and $K_2$ two hyperbolic knots in $S^3$. Let $N$ be the common covering of their complements, with a single cusp $C$. Let $G_1$ and  $G_2$ be the two associated covering groups. Then the subgroup $G \subset Isom^+ (N)$ generated by $G_1$ and $G_2$ is finite and acts identically  on $H_1(C, \mathcal{Z})$, since $G_1$ and $G_2$ do. So the quotient orbifold $\mathcal{O} = N/G$ has a single torus cusp.  By \cite{GAW} and \cite[Lemma 4]{Re1} the coverings $S^3 \setminus K_1 \to \mathcal{O}$
and $S^3 \setminus K_2 \to \mathcal{O}$ are cyclic.  Hence Proposition \ref{prop:cyclicdef} shows that $S^3 \setminus K_1$ and $S^3 \setminus K_2$ are cyclically commensurable.

\end{proof}

The following result is a consequence of the fact that a knot complement has a unique $2$-fold covering.

\begin{lemma} \label{onetorus} 
Let $K$ be a hyperbolic knot. If $K$ is strongly invertible, $\mathcal{Z}_{K}$ is the unique $2$-fold covering of $\mathcal{O}(K) = {\bf H}^3/N^+(K)$ with a torus cusp, up to an orientation-preserving homeomorphism.
\end{lemma} 

 \begin{proof}
Set $n = |Z(K)|$ and let $D_n$ denote the dihedral group of order $2n$. By hypothesis, $\Isom^+(S^3 \setminus K) \cong D_n$. We have the following exact sequence:
$$ 1 \rightarrow \pi_1(S^3 \setminus K) \rightarrow \pi_1(\mathcal{O}(K) ) \stackrel{\varphi}{\rightarrow} D_{n} \rightarrow 1 .$$
Let $\mathcal{O}'$ be a two-fold cover of $\mathcal{O}(K)$ with a torus cusp. Then $\pi_1(\mathcal{O}')$ is an index 2 subgroup of $\pi_1(\mathcal{O}(K) )$ whose image by $\varphi$ is a subgroup $G$ of $D_{n}$ of index $1$ or $2$. It is evident that $\hbox{ker}(\varphi|\pi_1(\mathcal{O}')) =  \pi_1(S^3 \setminus K) \cap \pi_1(\mathcal{O}')$ and
$[\pi_1(S^3 \setminus K) : \hbox{ker}(\varphi|\pi_1(\mathcal{O}'))] [D_n : G] = 2$.

If $G= D_n$, then $[\pi_1(S^3 \setminus K) : \hbox{ker}(\varphi|\pi_1(\mathcal{O}'))] = 2$, so $\hbox{ker}(\varphi|\pi_1(\mathcal{O}')) = \pi_1(M)$ where $M$ is the unique $2$-fold cyclic cover of $S^3 \setminus K$. There is a regular cover $M \to \mathcal{O}'$ of group $G = D_n$. Thus a strong inversion $\sigma \in \Isom^+(S^3 \setminus K)$ lifts to an involution $\tilde \sigma$ of $M$.
Since $M$ has one end, which is a torus cusp, it is easy to see that $\tilde \sigma$ acts on its first homology by multiplication by $-1$. But then $\mathcal{O}'$ has an $S^2(2,2,2,2)$ cusp, contrary to our hypotheses. 

Thus $[D_n : G] = 2$, so $|G| = |Z(K)|$ and $\ker(\varphi|\pi_1(\mathcal{O}')) = \pi_1(S^3 \setminus K)$. Therefore $S^3 \setminus K$ covers $\mathcal{O}'$ regularly with group $G \leq D_n = \Isom^+(S^3 \setminus K)$. Since $\mathcal{O}'$ has a torus cusp, $G$ acts freely on $K$, and so as $|G| = |Z(K)|$, $G = Z(K)$. Thus $\mathcal{O}' = (S^3 \setminus K)/Z(K) = \mathcal{Z}_K$, which is what we needed to prove. 
\end{proof}

\begin{remark}
{\rm The method of proof of the previous lemma yields the following stronger result: } Let $K$ be a hyperbolic knot and $S^3 \setminus K \to \mathcal{O}$ a cover where $\mathcal{O}$ is an orientable $3$-orbifold with an $S^2(2,2,2,2)$ cusp. Then there is a unique $2$-fold cover $\mathcal{O'} \to \mathcal{O}$ such that $\mathcal{O'}$ has a torus cusp. 
\end{remark}

The following proposition and Proposition \ref{prop:cyclicdef}   immediately implies assertion (1) of Theorem \ref{thm:cyclic}. \\

\begin{proposition} \label{prop:same} 
Two hyperbolic knots $K$ and $K'$ without hidden symmetries are commensurable if and only if there is an orientation-preserving homeomorphism between $\mathcal{Z}_{K}$ and $\mathcal{Z}_{K'}$. In particular  $K$ and $K'$ are commensurable if and only if they are cyclically commensurable.
\end{proposition}

\begin{proof} If there is an orientation-preserving homeomorphism between $\mathcal{Z}_{K}$ and $\mathcal{Z}_{K'}$ then clearly $K$ and $K'$ are commensurable, and in fact cyclically commensurable by Proposition \ref{prop:cyclicdef}. We prove the converse by distinguishing two cases:

 \noindent{\bf a) $K$ is not strongly invertible}. Then $\Isom^+(S^3 \setminus K) = Z(K)$ and since $K$ has no hidden symmetries, $\omin = (S^3 \setminus K)/Z(K) = \mathcal{Z}_{K}$. In particular $\omin$ has a torus cusp. Hence if $K'$ is commensurable with $K$, $K'$ is not strongly invertible. It follows that $\mathcal{Z}_{K'}$ is orientation-preserving homeomorphic to $\omin = \mathcal{Z}_{K}$.

\noindent{\bf b) $K$ is strongly invertible}. In this case $\omin = (S^3 \setminus K)/\hbox{Isom}^+(S^3 \setminus K)$ has a flexible cusp with horospherical section $S^2(2,2,2,2)$. Hence any knot $K'$ commensurable with $K$ is strongly invertible. The result follows from Lemma \ref{onetorus} as $\mathcal{Z}_{K}$ and $\mathcal{Z}_{K'}$ are $2$-fold coverings of
$\omin$ with torus cusps.

Now suppose that $K$ and $K'$ are commensurable knots without hidden symmetries. The proof shows that $S^3 \setminus K$ and $S^3 \setminus K'$ each cyclically cover $\mathcal{Z}_K$. Thus $K$ and $K'$ are cyclically commensurable by Proposition \ref{prop:cyclicdef}.  

\end{proof}

The following theorem is a main step in our study. It immediately implies Theorem \ref{thm:cyclic}. Recall that the meridinal slope of $S^3 \setminus K$ projects to a slope $r(K)$ in the torus cusp of $\mathcal{Z}_{K}  = (S^3 \setminus K)/Z(K)$ 

\begin{theorem} \label{th:slopes} 
Suppose that $K$ is a hyperbolic knot and let $K'$ be a knot cyclically commensurable with $K$. \\
$(1)$ There is an orientation-preserving homeomorphism between $\mathcal{Z}_{K}$ and $\mathcal{Z}_{K'}$. \\
$(2)$ If $r(K)$ and $r(K')$ coincide under some orientation-preserving homeomorphism between $\mathcal{Z}_{K}$ and $\mathcal{Z}_{K'}$, then $K$ and $K'$ are equivalent knots. \\
$(3)$ If $f_{K'}: \mathcal{Z}_{K'} \to \mathcal{Z}_{K}$ is a homeomorphism and $r_{K'}$ is the slope in the cusp of $\mathcal{Z}_{K}$ determined by $f_{K'}(r(K'))$, then $\Delta(r(K), r_{K'}) \leq1$. 
\end{theorem}  

\begin{proof}[Proof of Theorem \ref{thm:cyclic}]
By Assertion (1) of Theorem \ref{th:slopes} we can fix an orientation-preserving homeomorphism $f_{K'}: \mathcal{Z}_{K'} \to \mathcal{Z}_{K}$ for each $K' \in \mathcal{CC}(K)$. Let $r_{K'}$ be the slope in the cusp of $\mathcal{Z}_{K}$ determined by $f_{K'}(r(K'))$. Assertion (3) implies that there are at most three slopes in the set $\{r_{K'}: K' \in \mathcal{CC}(K)\}$, while assertion (2) implies that the function which associates the slope $r_{K'}$ to $K' \in \mathcal{C}(K)$ is injective. Thus Theorem \ref{thm:cyclic} holds. 
\end{proof}

Assertion (1) of Theorem \ref{th:slopes} is the content of the following proposition:

\begin{proposition} \label{prop:same2} 
Two hyperbolic knots $K$ and $K'$  are cyclically commensurable if and only if there is an orientation-preserving homeomorphism between $\mathcal{Z}_{K}$ and $\mathcal{Z}_{K'}$.
\end{proposition}

\begin{proof}
By Proposition \ref{prop:cyclicdef} if the hyperbolic knots $K$ and $K'$  are cyclically commensurable then there is an orientation-preserving homeomorphism between the orbifolds $\mathcal{O}(K) = {\bf H}^3/N^+(K)$ and  $\mathcal{O}(K') = {\bf H}^3/N^+(K')$. Then the proof is the same as the proof of Proposition \ref{prop:same} by considering
$\mathcal{O}(K)$ instead of $\omin$.
\end{proof}

Assertion (2) of Theorem \ref{th:slopes} is given by the following lemma:

\begin{lemma}\label{lem:slope} Let $K$ and $K'$ be two hyperbolic cyclically commensurable knots. If $r(K)$ and $r(K')$ coincide under some orientation-preserving homeomorphism between $\mathcal{Z}_{K}$ and $\mathcal{Z}_{K'}$, then $K$ and $K'$ are equivalent knots. 
\end{lemma}

\begin{proof} Suppose that $r(K)$ and $r(K')$ coincide under some homeomorphism $\mathcal{Z}_{K} \to \mathcal{Z}_{K'}$. Then we have an induced homeomorphism $f: (\mathcal{Z}_K(r(K)), \mathcal{Z}_K) \to (\mathcal{Z}_{K'}(r(K')), \mathcal{Z}_{K'})$. 
By construction, $\mathcal{Z}_K(r(K)) \cong S^3/Z(K)$ so 
$$\pi: S^3 \to S^3/Z(K) = \mathcal{Z}_K(r(K))$$ 
is a universal cover. In the same way 
$$\pi': S^3 \to S^3/Z(K') = \mathcal{Z}_{K'}(r(K'))$$ 
is a universal cover. 
Since universal covers are unique up to covering equivalence, there is a homeomorphism (preserving orientation) $\tilde f: S^3 \to S^3$ such that 
$\pi' \circ \tilde f = f \circ \pi$. In particular, 
$$\tilde f(S^3 \setminus K) = \tilde f(\pi^{-1}(\mathcal{Z}_K)) = \pi'^{-1}(f(\mathcal{Z}_K))) = \pi'^{-1}(\mathcal{Z}_{K'}) = S^3 \setminus K' .$$
Thus the complement of $K$ is orientation-preserving homeomorphic to the complement of $K'$, so $K$ is equivalent to $K'$ \cite{GL}.
\end{proof}

With the notations of Lemma \ref{lem:slope}, Assertion (3) of Theorem \ref{th:slopes} is the content of the following lemma:
    
\begin{lemma}\label{lem:intersection} Let $K$ and $K'$ be two cyclically commensurable knots and $f: \mathcal{Z}_{K'} \to \mathcal{Z}_{K}$ a homeomorphism\footnote{We do not assume that $f$ preserves orientation.}. Then $\Delta(r(K), r_{K'}) \leq1$ where $r_{K'}$ is the slope in the cusp of $\mathcal{Z}_K$ corresponding to $f(r(K'))$. 
\end{lemma}

\begin{proof}

Set 
$$\mathcal{Z}_K^0  = \mathcal{Z}_K \setminus N(\Sigma(\mathcal{Z}_K))$$ where $N(\Sigma(\mathcal{Z}_K))$ denotes a small, open tubular neighborhood of $\Sigma(\mathcal{Z}_K)$. Then $\mathcal{Z}_K^0$ has no singularities. Since $\Sigma(\mathcal{Z}_K)$ is a geodesic link in the hyperbolic orbifold $\mathcal{Z}_K$, $\mathcal{Z}_K^0$ admits a complete, finite volume, hyperbolic structure \cite{knot, link}. 

By the geometrization of finite group actions \cite{orbifoldthm, Morgan}, we can suppose that $Z(K)$ and $Z(K')$ act orthogonally on $S^3$. It follows that both Dehn fillings of  the torus cusp of $\mathcal{Z}_K$ along the slopes $r(K)$ and $r_{K'}$ give orbi-lens spaces $\mathcal{L}_K = \mathcal{Z}_K(r(K)) = S^3/Z(K)$ and $\mathcal{L}' = \mathcal{Z}_K(r_{K'}) \cong \mathcal{Z}_{K'}(r(K')) = S^3/Z(K')$. By Corollary \ref{equivorbilens}, $|\mathcal{L}_K|$  and $|\mathcal{L}'|$ are  lens spaces, possibly $S^3$. Moreover  
the singular set $\Sigma(\mathcal{L}_K)$, resp. $\Sigma(\mathcal{L} ')$, is either empty or a sublink of 
the union of the cores of the two solid tori in a genus 1 Heegaard splitting of $|\mathcal{L}_K|$, resp. $|\mathcal{L}'|$. 
Since $\mathcal{Z}_K^0(r(K)) = \mathcal{L}_K \setminus N(\Sigma(\mathcal{Z}_K)) 
= \mathcal{L}_K \setminus N(\Sigma(\mathcal{L}_K ))$, we have:

$$\mathcal{Z}_K^0(r(K))  \cong \left\{ 
\begin{array}{ll}
|\mathcal{L}_{K}| &  \hbox{ if } \vert \Sigma(\mathcal{Z}_K) \vert = 0 \\ 
S^1 \times D^2 &  \hbox{ if } \vert \Sigma(\mathcal{Z}_K) \vert = 1 \\
S^1 \times S^1 \times [0,1] & \hbox{ if } \vert \Sigma(\mathcal{Z}_K) \vert = 2
\end{array} \right. 
$$

In the same way:

$$\mathcal{Z}_K^0(r_{K'})  \cong \left\{ 
\begin{array}{ll}
|\mathcal{L}'| &  \hbox{ if } \vert \Sigma(\mathcal{Z}_K) \vert = 0 \\ 
S^1 \times D^2 &  \hbox{ if } \vert \Sigma(\mathcal{Z}_K) \vert = 1 \\
S^1 \times S^1 \times [0,1] & \hbox{ if } \vert \Sigma(\mathcal{Z}_K) \vert = 2
\end{array} \right. 
$$

 One can choose slopes on the components $\partial N(\Sigma(\mathcal{Z})) \subset \partial \mathcal{Z}_K^0$ such that $M$, the manifold obtained by Dehn filling $\mathcal{Z}_K^0$ along these slopes, is hyperbolic. It follows from above that $M(r(K))$ and $M(r')$ have cyclic fundamental groups, so the cyclic surgery theorem \cite{CGLS} implies that $\Delta(r(K), r_{K'}) \leq 1$.
 \end{proof}

This completes the proof of Theorem \ref{th:slopes}, and therefore of Theorem \ref{thm:cyclic}. 

We have the following consequence of the proof.  A good orbifold is an orbifold which is covered by a manifold.  

\begin{scholium} Let $M$ be a hyperbolic orbifold with a single torus cusp.  If $M(r_1)$ and $M(r_2)$ yield good orbifolds with cyclic orbifold fundamental group, then $\Delta(r_1,r_2) \leq 1$.  In particular, there are at most 3 such slopes.
\end{scholium} 

\begin{proof}

Suppose that the group $\pi_1^{\orb}(M(r_1))$ is finite cyclic.  Then the universal cover is $S^3$ and $M$ is the complement of a knot in an orbi-lens space, and the result follows from the proof of Lemma \ref{lem:intersection}.

Suppose $\pi_1^\orb(M(r_	1))$ is infinite cyclic.  Since its universal cover is a manifold and its fundamental group has no torsion, $M(r_1)$ is a manifold and hence $M$ is a hyperbolic manifold. The result follows from the Cyclic Surgery Theorem \cite{CGLS}. 
\end{proof}

The analysis of the action on the knot complement by a cyclic group of symmetries as in Lemma \ref{lem:intersection} above along with an observation of M. Kapovich yields the following characterisation of the minimal element in the commensurability class of a knot complement. 

\begin{corollary} \label{ball} If $\mathcal{O}_{\text{min}}(K)$ is the minimal element of a non-arithmetic commensurability class which contains a knot complement $S^3 \setminus K$ then the underlying space of $\mathcal{O}_{\text{min}}(K)$ is either an open ball or the complement of a knot in a lens space.  \end{corollary}

\begin{proof}
Let $\mathcal{O}_{\text{min}}(K)$ be the minimal element of the commensurability class, and $\mathcal{\hat O}_{\text{min}}(K)$ the associated orbifold with boundary obtained by truncating along the cusp.  Since the boundary of $S^3 \setminus N(K)$ is a torus, $\partial \mathcal{\hat O}_{\text{min}}(K)$ is a closed orientable Euclidean 2-orbifold, which implies that it either a torus or has underlying space $S^2$.  When $\partial \mathcal{\hat O}_{\text{min}}(K)$ is a torus, the covering is a regular cyclic covering by \cite{GAW} and \cite[Lemma 4]{Re1}.  Therefore, our analysis in Lemma \ref{lem:intersection} implies that the underlying space of $\mathcal{\hat O}_{\text{min}}(K)$ is a lens space with a regular neighborhood of a knot removed, and that $|\mathcal{O}_{\text{min}}(K)|$ is the complement of a knot in a lens space.  The case when $|\partial \mathcal{\hat O}_{\text{min}}(K)|$ is $S^2$ is an observation of M. Kapovich.  There is a map which is the composition $S^3 \setminus N(K) \rightarrow \mathcal{\hat O}_{\text{min}}(K) \rightarrow |\mathcal{\hat O}_{\text{min}}(K)|$. The image of $\pi_1(S^3 \setminus N(K))$ under the induced homomorphism is trivial, as $\pi_1(S^3 \setminus N(K))$ is normally generated by a meridian.  Therefore if $|\mathcal{\hat O}_{\text{min}}(K)|$ has any non-trivial cover (such as the universal cover) the above map 
$S^3 \setminus N(K) \rightarrow |\mathcal{\hat O}_{\text{min}}(K)|$ lifts to this cover, which is a contradiction as any non-trivial cover of a manifold with boundary $S^2$ has multiple boundary components.  Therefore $|\mathcal{\hat O}_{\text{min}}(K)|$ has trivial fundamental group and by work of Perelman \cite{Morgan} it is a ball. Hence $\mathcal{O}_{\text{min}}(K)$ has underlying space an open ball.  

\end{proof} 

By \cite[Main Theorem]{Ar}, see also \cite{Neu}, $\Gamma \in \Isom^+({\bf H}^3)$ is generated by rotations exactly when the underlying space of ${\bf H}^3 / \Gamma$ is simply-connected.  Therefore we have the following corollary of Corollary \ref{ball}. 

\begin{corollary} A non-invertible hyperbolic knot $K$ has a hidden symmetry if and only if its group $\pi_1(S^3 \setminus K)$ is commensurable with a Kleinian group generated by rotations.
\end{corollary}

The following proposition is a  consequence of the proof of  Theorem \ref{th:slopes}. It states that a hyperbolic knot $K$   is not unique in its cyclic commensurability class if and only if $\bar K \subset \mathcal{L}_K$ admits a non-trivial orbi-lens space surgery. More precisely:

\begin{proposition} \label{prop:commensurable} 
A commensurability class contains cyclically commensurable knot complements $S^3 \setminus K$ and $S^3 \setminus K'$ where $K' \ne K$ if and only if  it contains the complement of a knot $\bar K$ in an orbi-lens space $\mathcal{L}$ such that $\bar K$ is primitive in $\mathcal{L}$ and $\mathcal{L}$ admits a non-trivial orbi-lens space surgery $\mathcal{L}'$ along $\bar K$.  We may take $\mathcal{L} \setminus \bar K$ to be $\mathcal{Z}_k$, with slopes $r(K)$ and $r'$ yielding the lens spaces $\mathcal{L}$ and $\mathcal{L}'$ respectively.  If  $\pi':S^3 \rightarrow \mathcal{L}'$ is the universal covering and $\bar K' \subset \mathcal{L}'$ is the core of the $r'$-Dehn filling of $\mathcal{Z}_K$, then $K' = \pi ^{-1}(\bar K'). $
\end{proposition}

This result gives a way of constructing every knot cyclically  commensurable with $K$. Since the only non-arithmetic knots known to admit hidden symmetries are the two commensurable dodecahedral knots of Aitchison and Rubinstein \cite{AR}, all the other pairs of commensurable hyperbolic knots constructed so far can be obtained from the construction given in  Proposition \ref{prop:commensurable}.

\begin{proof}[Proof of Proposition \ref{prop:commensurable}] 
We continue to use the notation developed in the proof of Theorem \ref{th:slopes}. Suppose a commensurability class $\mathcal{C}$ contains cyclically commensurable knot complements $S^3 \setminus K$ and $S^3 \setminus K'$.  By the proof of Theorem \ref{prop:cyclicdef} the quotients $\mathcal{Z}_k$ and $\mathcal{Z}_{K'}$ are homeomorphic.  By the proof of Theorem \ref{th:slopes}, there are distinct slopes $r(K)$ and $r_{K'}$,  of $\mathcal{Z}_k$  such that filling along these slopes produces lens spaces $\mathcal{L}_K$ and $\mathcal{L}_{K'}$ respectively.  Also, the preimages of the surgery core $\bar K$ in the universal covers of $\mathcal{L}_K$ and $\mathcal{L}_{K'}$ are the knots $K \subset S^3$ and $K' \subset S^3$.  Since $K$ is a knot, $\bar K$ is primitive in $\mathcal{L}_K$.   Thus  $\mathcal{Z}_K$ satisfies the conclusions of the theorem.

Suppose that a commensurability class $\mathcal{C}$ contains the complement of a knot in an orbi-lens space $\mathcal{L} \setminus \bar{K}$ where $\bar K$ is primitive in $\mathcal{L}$ and $\bar K$ admits a non-trivial orbi-lens space surgery.  Then by primitivity, the pre-image of $\bar K$ in the universal cover $S^3$ of $\mathcal{L}$ is a knot $K$.  Since the covering group $S^3 \rightarrow \mathcal{L}$ is cyclic, $S^3 \setminus K$ cyclically covers $\mathcal{L} \setminus \bar K \cong \mathcal{O}$.  Let $r_K$ be the projection of the meridinal slope of $S^3 \setminus K$.  Denote the non-trivial orbi-lens space filling of $\mathcal{L} \setminus \bar K$ by $\mathcal{L}'$ and the filling slope by $r_{K'}$.  By the proof of Lemma \ref{lem:intersection}, $\Delta(r_K, r_{K'}) \leq 1$. Thus a representative curve for $r_{K'}$ is isotopic to $\bar K$ in $\mathcal{L}$. It follows that representative curves for $r_K$ and $r_{K'}$ carry the first homology of $\mathcal{L} \setminus \bar K$. Thus the core $\bar K'$  of the $r_{K'}$-Dehn filling solid torus in $\mathcal{L}'$ carries a generator of $H_1(\mathcal{L}')$ and therefore the pre-image of $\bar K'$ in the universal cover of $\mathcal{L}'$  is a knot in $S^3$. Furthermore, $S^3 \setminus K'$ cyclically covers $\mathcal{O} \cong \mathcal{L}' \setminus \bar K' \cong \mathcal{L} \setminus \bar K$. Therefore, $\mathcal{C}$ contains the cyclically commensurable knots $S^3 \setminus K$ and $S^3 \setminus K'$. Suppose that $K$ is equivalent to $K'$.  An orientation-preserving homeomorphism $S^3 \setminus K \to S^3 \setminus K'$ induces an orientation-preserving homeomorphism $f: \mathcal{L} \setminus \bar K \to \mathcal{L}' \setminus \bar K'$. It is evident that $f(r(K)) = r(K')$. By construction we have an orientation-preserving homeomorphism $g: \mathcal{L} \setminus \bar K \to \mathcal{L}' \setminus \bar K'$ such that $g(r') = r(K')$. Thus $h = g^{-1} \circ f: \mathcal{L} \setminus \bar K \to \mathcal{L} \setminus \bar K$ is an orientation-preserving homeomorphism such that $h(r(K)) = r'$. But this is impossible as Lemma \ref{lem:slopes well-defined} would then imply that $r'  = r(K)$. Thus $K$ and $K'$ are distinct knots by Theorem \ref{th:slopes}. 
\end{proof}

This suggests the following conjecture: 

\begin{conjecture}  \label{conj:rigid} A rigid commensurability class does not contain cyclically commensurable hyperbolic knot complements. 
\end{conjecture} 

Theorem \ref{th:periodic} and Proposition \ref{prop:commensurable} reduce the characterisation of hyperbolic knots $K \subset S^3$ such that $|\mathcal{CC}(K)| \geq 2$  to the case where $Z(K)$ acts freely on $S^3$ and to the construction of all primitive knots in a lens space with a non-trivial lens space surgery.  We remark that the situation is completely understood for the case of orbi-lens spaces:

 

\begin{proposition}  Let $\bar K$ be  a primitive hyperbolic knot in an orbi-lens space $L$ with non-trivial ramification locus $\Sigma(L)$.  If a non-trivial Dehn surgery along $\bar K$  produces an orbi-lens space, then $K$ is a Berge-Gabai knot in $L \setminus N(\Sigma(L))$. 
\end{proposition}

\begin{proof}  Let $V_1 \cup V_2$ be the Heegaard splitting of $L$ where $V_1$ is a regular neighborhood of $\Sigma(L)$ and $\bar K \subset V_2$.  Assume non-trivial surgery along $\bar K$ in $L$ yields an orbi-lens space $L'$.  By removing neighborhoods of the ramification loci in $L$ and $L'$, we see that non-trivial surgery along $\bar K$ in $V_2$ yields a solid torus.  Then by Definition \ref{def:bg} $\bar K$ is a Berge-Gabai knot in $V_2 = L \setminus N(\Sigma(L)).$
\end{proof}

\section{Unwrapped $1$-bridge braids} \label{unwrapped}

In this section we prove Theorem \ref{th:periodic} which characterizes all periodic hyperbolic knots such that $|\mathcal{CC}(K)| \geq 2$.

Recall that a {\it $1$-bridge braid} in a solid torus $V$ is a braid in $V$ which is $1$-bridge with respect to some boundary-parallel torus in $\hbox{int}(V)$. Connected $1$-bridge braids have been classified in \cite{Gabai3}. 

A {\it cosmetic surgery slope} of a knot in a $3$-manifold $W$ is a slope on the boundary of the exterior of the knot whose associated surgery yields a manifold homeomorphic to $W$. We say that $K$ has a {\it non-trivial cosmetic surgery} if it has such a slope which is distinct from the knot's meridian. The following proposition is a consequence of work of Gabai and Gordon-Luecke. 

\begin{proposition} \label{Berge-Gabai} If a hyperbolic knot $K$ in $V \cong S^1 \times D^2$ or $V \cong S^1 \times S^1 \times I$ admits a non-trivial cosmetic surgery, then $V \cong S^1 \times D^2$ and $K$ is a $1$-bridge braid. 
\end{proposition}

\begin{proof} First assume that $K$ is a hyperbolic knot in $V \cong S^1 \times D^2$. Gabai \cite{Gabai1} has shown that any knot in a solid torus which admits a non-trivial cosmetic surgery is either contained in a $3$-ball or is a $0$-bridge braid or is a $1$-bridge braid. In our case, hyperbolicity rules out the first two cases. Thus $K$ is a $1$-bridge braid. 

The case where  $K$ is a hyperbolic knot in $V \cong S^1 \times S^1 \times I$ is ruled out by the following lemma:

\begin{lemma}\label{lem:product} A hyperbolic knot $K$ in $V \cong S^1 \times S^1 \times I$ admits no non-trivial cosmetic surgery.
\end{lemma}

\begin{proof} Assume that there is a non-trivial cosmetic surgery $r$ for $K$. Then $r$ is a non-trivial cosmetic surgery slope when $K$ is considered as a knot in any Dehn filling of $V$ along $T^2 \times {0}$. Choose such a filling in which $K$ remains hyperbolic. The previous argument then implies that $K$ is not homologically trivial in the Dehn filling of $V$, and therefore not in $V$ as well. Then   there is an essential simple closed curve $C \subset T^2 \times {0}$ such that the class in $H_1(V)$ carried by $K$ is an integral multiple of that carried by $C$. Since the algebraic intersection of $K$ with the properly embedded, essential annulus $A = C \times I \subset V$ is nul, $A$ defines a homology class $[A] \in H_2(V\setminus K, \partial V) \cong \mathbb Z$. Let $(F, \partial F) \subset (V\setminus K, \partial V)$ be a norm minimizing surface representing the homology class $[A]$. By a result of Gabai \cite[Corollary]{Gabai2}, $F$ remains norm minimizing in all manifolds obtained by Dehn surgeries along $K$ except at most one. Since two such surgeries yield manifolds homeomorphic to $S^1 \times S^1 \times I$, $F$ must be an essential annulus, contrary to the hypothesis that $K$ is hyperbolic in $V$. Thus the lemma holds. 
\end{proof} 
This completes the proof of Proposition \ref{Berge-Gabai}. \end{proof} 

Recall the hyperbolic manifold 
$$\mathcal{Z}_K^0 = \mathcal{Z}_K \setminus N(\Sigma(\mathcal{Z}_K))$$ 
defined in the proof of Lemma \ref{lem:intersection}. It follows from this proof that if $|\mathcal{CC}(K)| > 1$ and $|\Sigma(\mathcal{Z}_K)| \geq 1$, then the core $\bar K$ of the Dehn filling $\mathcal{Z}_K^0 (r(K)) \cong S^1 \times D^2$ or $S^1 \times S^1 \times I$ admits a non-trivial cosmetic surgery. Hence Proposition \ref{Berge-Gabai} immediately implies the following corollary: 

\begin{corollary} \label{s=2}
If $K$ is a periodic hyperbolic knot and $|\Sigma(\mathcal{Z}_K)| = 2$, then $|\mathcal{CC}(K)| = 1$. In particular, if $K$ has no hidden symmetry 
$|\mathcal{C}(K)| = 1$. 
\hfill $\Box$ 
\end{corollary}

This result implies assertion (1) of Theorem \ref{th:periodic}. Next we examine the case $|\Sigma(\mathcal{Z}_{K})| = 1$. 

\begin{definition}  \label{def:bg}
{\rm A {\it Berge-Gabai knot in a solid torus} is a $1$-bridge braid in a solid torus which admits a non-trivial cosmetic surgery slope.}
\end{definition}

The {\it winding number} of a Berge-Gabai knot in a solid torus is the braid index of its associated $1$-bridge braid.

Berge-Gabai knots and their cosmetic surgery slopes have been classified. See \cite{Bergetori}, \cite{Gabai3}. Moreover, it follows from the description given in \cite{Gabai3} that these knots can be embedded in $S^3$ as homogeneous braids and hence as fibred knots by Stallings \cite{Sta1}.

\begin{definition} $\;$ \\ 
{\rm (1) Let $w, p, q, a$ be integers with $w, a, p \geq 1$ and $\gcd(p, q) = \gcd(w, ap) = 1$. A {\it Berge-Gabai knot $\bar K$ of winding number $w$} in $\mathcal{L}(p,q; a)$ consists of a knot $\bar K \subset \mathcal{L}(p,q; a)$ and a genus one Heegaard splitting $V_1 \cup V_2$ of $|\mathcal{L}(p,q; a)|$ such that $\bar K$ is a Berge-Gabai knot of winding number $w$ in $V_1$ and $\Sigma(\mathcal{L}(p,q;a))$ is a closed submanifold of the core of $V_2$. \\ 
(2) A {\it $(p,q;a)$-unwrapped Berge-Gabai knot} in $S^3$  is the inverse image of a Berge-Gabai knot in $\mathcal{L}(p,q; a)$ under the universal cover $S^3 \to \mathcal{L}(p,q; a)$.}
\end{definition}

Note that the inverse image in $S^3$ of a Berge-Gabai knot in $\mathcal{L}(p,q; a)$ is connected if and only if the winding number $w$ of the knot is coprime to $ap$.

\begin{lemma} \label{coprime} 
Suppose that $\bar K$ is a Berge-Gabai knot of winding number $w$ in $\mathcal{L}(p,q; a)$ where $p \geq 1$ is coprime with $w$. Let $V_1 \cup V_2$ be a Heegaard splitting of $L(p,q)$ where $\bar K \subset \hbox{int}(V_1)$ and $\Sigma(\mathcal{L}(p,q;a))$ is a closed submanifold of the core of $V_2$. If $r$ is a non-trivial cosmetic surgery slope of $\bar K$ considered as a knot in $V_1$, then $\bar K(r) \cong \mathcal{L}(p', q'; a)$ where $\gcd(p, p') = 1$. 
\end{lemma}

\begin{proof} 

It is clear that $|\bar K(r)|$ has Heegaard genus one, so is $L(p',q')$ for some $p' \geq 0$. (We take the convention that $L(0, q') \cong S^1 \times S^2$.) We must show $p'$ is non-zero and relatively prime to $p$.

Let $W$ be the exterior of $\bar K$ in $V_1$ and write $\partial W = T_0 \cup T_1$ where $T_1 = \partial V_1$ and $T_0$ is the boundary of a tubular neighborhood of $\bar K$. There are bases $\mu_0, \lambda_0$ of $H_1(T_0)$ and $\mu_1, \lambda_1$ of $H_1(T_1)$ such that $\mu_0$ is a meridian of $\bar K$, $\mu_1$ is a meridian of $V_1$, and $\mu_1 = w \mu_0, \lambda_0 = w \lambda_1$ in $H_1(W)$. 

It is shown in Lemma 3.2 of \cite{Gabai3} that $r = \pm(m \mu_0 + \lambda_0)$ where $\gcd(m, w) = 1$. A homological calculation (see Lemma 3.3 of \cite{Go}) shows that $\mu_1(r)$, the meridian slope of the solid torus $(V_1, \bar K)(r)$, is given by $\mu_1(r) = m \mu_1 + w^2 \lambda_1$. By hypothesis, $q \mu_1 + p \lambda_1$ is the meridian of $V_2$ and therefore 
$$p' = \Delta(\mu_1(r), q \mu_1 + p \lambda_1) = 
\Delta(m \mu_1 + w^2 \lambda_1, q \mu_1 + p \lambda_1) = |mp - qw^2| .$$
Since $p$ is coprime to $q$ and $w^2$, it is coprime to $|mp - qw^2|$, and since $p \geq 1$ and $\gcd(m, w) = 1$, $|mp - qw^2| \ne 0$. Thus the lemma holds. 
\end{proof}

Next we characterize periodic hyperbolic knots $K$  such that  $|\Sigma(\mathcal{Z}_K)| = 1$ and $|\mathcal{CC}(K)| \geq 2$. This will finish the proof of Theorem \ref{th:periodic}.

\begin{proposition}\label{prop:berge-gabai}
Let $K$ be a hyperbolic knot in $S^3$. \\ 
$(1)$  If $K$ is periodic such that $|\Sigma(\mathcal{Z}_K)| = 1$ and $|\mathcal{CC}(K)| \geq 2$ then \\
\indent \hspace{.3cm} $(a)$ $S^3/Z(K) = \mathcal{L}(p,q; a)$, where $ap = |Z(K)|$ and  the image $\bar K$ of $K$ in $\mathcal{L}(p,q; a)$ is \\
\indent \hspace{.9cm} a Berge-Gabai knot of winding number prime to $|Z(K)|$. Thus $K$ is the $(p,q;a)$-\\
\indent \hspace{.9cm} unwrapped Berge-Gabai knot associated to the core of the surgery torus in $\mathcal{Z}_K^0(r(K))$\\
\indent \hspace{.9cm} $\cong S^1 \times D^2$. \\
\indent \hspace{.3cm} $(b)$  $K$ is strongly invertible. \\
\indent \hspace{.3cm} $(c)$ each $K' \in \mathcal{CC}(K) \setminus K$ is a $(p',q';a)$-unwrapped Berge-Gabai knot associated  \\
\indent \hspace{.9cm} to the core of the surgery solid torus in $\mathcal{Z}_K^0(r')$ where $|Z(K')| = ap', \gcd(p, p') = 1$,  \\
\indent \hspace{.9cm} and $r' = f(r(K'))$ where $f: \mathcal{Z}_{K'} \to \mathcal{Z}_K$ is an orientation-preserving homeomor- \\
\indent \hspace{.9cm} phism.\\ 
$(2)$ If $K$ is a $(p,q;a)$-unwrapped Berge-Gabai knot, then $|\mathcal{CC}(K)| \geq 2$. 
\end{proposition}

This result holds for a periodic hyperbolic knot $K$ without hidden symmetries and any $K' \in \mathcal{C}(K) \setminus K$.

\begin{proof}
First suppose that $K$ is a knot without hidden symmetries such that $|\Sigma(\mathcal{Z}_K)| = 1$ and $|\mathcal{CC}(K)| > 1$. 
Corollary \ref{equivorbilens} shows that $S^3/Z(K)$ is an orbi-lens space $\mathcal{L}(p,q; a)$ where $ap = |Z(K)|$. Let $\bar K$ be the image of $K$ in $\mathcal{L}(p,q; a)$. There is a genus one Heegaard splitting $V_1 \cup V_2$ of $L(p,q)$ such that $V_1 = \mathcal{Z}_K^0(r(K))$ and $V_2$ is a regular neighborhood of $\Sigma(\mathcal{L}(p,q; a)) = \Sigma(\mathcal{Z}_K)$. It follows from Theorem \ref{th:slopes} and the proof of Lemma \ref{lem:intersection} that for $K \ne K' \in \mathcal{C}(K)$, the image of $r(K')$ in the cusp of $\mathcal{Z}_K$ under a homeomorphism $\mathcal{Z}_{K'} \to \mathcal{Z}_K$ is a non-trivial cosmetic surgery slope of $\bar K$ in $V_1$. Hence, $\bar K$ is a Gabai-Berge knot in $V_1$, and as its inverse image in $S^3$ is $K$, it has winding number coprime to $ap = |Z(K)|$. Thus $K$ is a $(p,q;a)$-unwrapped Berge-Gabai knot. 

Note that as $\bar K$ is $1$-bridge braid in $V_1$, it lies on a genus $2$ Heegaard surface of $L(p,q)$ (c.f. the proof of Theorem \ref{th_fibred}). It follows that $\mathcal{L}(p,q;a)$ admits an orientation-preserving involution which reverses the orientation of $\bar K$. Hence $\mathcal{Z}_K = (S^3 \setminus K)/Z(K)$ is not minimal in its commensurability class. It follows that $\Isom^+(S^3 \setminus K) \ne Z(K)$, so $K$ is strongly invertible. 

Consider $K' \in \mathcal{C}(K) \setminus K$. Since the hypotheses hold for $K'$ in place of $K$, we see that $K'$ is the $(p',q';a)$-unwrapped Berge-Gabai knot associated to the core $\bar K'$ of the surgery solid torus in $\mathcal{Z}_K^0(r')$ where $|Z(K')| = ap'$ 
and $r'$ is the image in the cusp of $\mathcal{Z}_K$ of $r(K')$ under an orientation-preserving homeomorphism $f: \mathcal{Z}_{K'} \to \mathcal{Z}_K$ (c.f. Proposition \ref{prop:same}). Lemma \ref{coprime} implies that $\gcd(p, p') = 1$. This completes the proof of assertion (1).

Next we prove assertion (2). Suppose that $K$ is a $(p,q;a)$-unwrapped Berge-Gabai knot. If $ap = 1$, then $\mathcal{L}(p, q; a) = S^3$. Lemma \ref{coprime} and \cite{KM} show that $K$ has a slope $r$ such that $K(r)$ is a lens space whose fundamental group is non-trivial. This case of assertion (2) then follows from Proposition \ref{prop:commensurable}.

If $ap > 1$, there is a Berge-Gabai knot $\bar K$ in $\mathcal{L}(p,q;a)$ whose inverse image under the universal cover $S^3 \to \mathcal{L}(p,q;a)$ is $K$. Since Berge-Gabai knots in solid tori admit non-trivial cosmetic surgeries, Lemma \ref{coprime} implies that there is a non-trivial slope $r$ of $\bar K$ such that $\bar K(r) \cong \mathcal{L}(p',q';a)$ where $\gcd(p,p') = 1$. This final case of assertion (2) now follows from Proposition \ref{prop:commensurable}.
\end{proof}

We conclude this section with the observation that the characterisation in Proposition \ref{prop:berge-gabai} allows us to show that hyperbolic knot complements with the same volume are not cyclically commensurable. 

\begin{proposition}\label{prop:volume}
Let $K$ be a hyperbolic knot with $|\mathcal{CC}(K)| \geq 2$. Then: \\
$(1)$ the volume of $K$ is different from that of any $K' \in \mathcal{CC}(K) \setminus K$.\\
$(2)$ the only mutant of $K$ contained in $\mathcal{CC}(K)$ is $K$. \\
$(3)$ if $K$ is commensurable with its mirror image, it is amphichiral. 
\end{proposition}

This result holds for a hyperbolic knot $K$ without hidden symmetries and any $K' \in \mathcal{C}(K) \setminus K$. 
\begin{proof}

First we prove that if $K' \in \mathcal{CC}(K)$ is distinct from $K$, then the cyclic groups $Z(K)$ and $Z(K')$ have distinct orders. This will imply that $K$ and $K'$ have distinct volumes since  $\hbox{vol}(S^3 \setminus K) = |Z(K)| \hbox{vol}(\mathcal{Z}_K) 
\ne |Z(K')| \hbox{vol}(\mathcal{Z}_K)=\hbox{vol}(S^3 \setminus K')$.

Suppose that $Z(K)$ acts freely on $S^3$. Then $\mathcal{L}_K$ is a lens space of the form $L(c, d)$ where $c = |Z(K)|$. Let $M$ denote the exterior of $\bar K$ in $L(c, d)$ and note that as $\bar K$ is primitive, $H_1(M) \cong \mathbb Z$. Hence there is a basis $\bar \mu, \bar \lambda$ of $H_1(\partial M)$ such that the image of $\bar \mu$ in $H_1(M)$ generates while the image of $\bar \lambda$ is trivial. Clearly, the meridinal slope of $\bar K$ represents $c \bar \mu + e \bar \lambda$ in $H_1(\partial M)$ for some integer $e$.
Similarly $\mathcal{L}_{K'}$ is a lens space $L(c',d')$ where $c' = |Z(K')|$, so the meridinal slope of $\bar K'$ represents $c' \bar \mu + e' \bar \lambda$. The cyclic surgery theorem \cite{CGLS} implies that $\pm 1 = c e' - e c'$, so $\gcd(c, c') = 1$. Note that we cannot have $c = c' = 1$ as otherwise some non-trivial surgery on a hyperbolic knot in $S^3$ would yield $S^3$, contrary to  \cite{GL}. Thus $c \ne c'$, so the proposition holds when $Z(K)$ acts freely on $S^3$.

Suppose next that $Z(K)$ does not act freely on $S^3$. By Proposition \ref{prop:berge-gabai}, $K$ is a $(p,q;a)$-unwrapped Berge-Gabai knot and $K'$ is a $(p',q';a)$-unwrapped Berge-Gabai knot where $p$ and $p'$ are coprime by Lemma \ref{coprime} and $a > 1$. Since $|Z(K)| = ap$ and $|Z(K')| = ap'$, it follows that $|Z(K)|\ne |Z(K')|$ unless $p = p' =1$. Assume $p = p' = 1$. There is a Heegaard splitting $|\mathcal{L}(1,q;a)| = V_1 \cup V_2$ where the singular set of $\mathcal{L}(1,q;a)$ is the core $C_2$ of $V_2$ and a hyperbolic Berge-Gabai knot $\bar K \subset V_1 \subset \mathcal{L}(1,q; a)$ such that $K$ is the inverse image of $\bar K$ in $S^3$. Since $C_2$ is unknotted in $|\mathcal{L}(1,q;a)| \cong S^3$, Corollary 3.5 of \cite{Gabai3} implies that its image is knotted in $|\mathcal{L}(1,q;a)| \cong S^3$. But this contradicts the fact that the image of $C_2$ in $|\mathcal{L}(1,q';a)|$ is the core of a Heegaard solid torus. Hence we cannot have $p = p' = 1$. This completes the proof that $Z(K)$ and $Z(K')$ have distinct orders and therefore that $K$ and $K'$ have distinct volumes.

Since mutant hyperbolic knots have the same volume, $K$ and $K'$ cannot be mutant. Similarly hyperbolic knots which are mirror images of each other have the same volume so as $K' \ne K$, $K'$ cannot be the mirror image of $K$.
\end{proof}

\section{Fibred knots} \label{fibred}


In this section we prove that any hyperbolic knot without hidden symmetries and with $|\mathcal{CC}(K)| \geq 2$ is fibred (Theorem \ref{th:properties}(1)). 

We divide the proof of Theorem \ref{th:properties}(1) into two cases according to whether $K$ is periodic or not.

\subsection{$K$ is periodic} 

Here we prove a fibering theorem for $1$-bridge braid exteriors and apply it to deduce the periodic case of Theorem \ref{th:properties}(1). 

\begin{theorem} \label{th_fibred} 
Let $K$ be a $1$-bridge braid on $n$ strands in a solid torus $V$. For any essential simple closed curve $C$ on $\partial V$ whose algebraic winding number in $V$ is coprime to $n$ there is a locally trivial fibring of the exterior of $K$ in $V$ by surfaces whose intersection with $\partial V$ has $n$ components, each a curve parallel to $C$.  
\end{theorem}

\begin{corollary} \label{cor:ubg periodic} 
An unwrapped Berge-Gabai knot is a fibred knot.
\end{corollary}

\begin{proof}[Proof of Corollary \ref{cor:ubg periodic}] 
Let $K$ be an unwrapped Berge-Gabai knot in $S^3$. Then $K$ is the inverse image in $S^3$ of a Berge-Gabai knot $\bar K \subset \mathcal{L}(p,q; a)$ of winding number $n$, say, under the universal cover $S^3 \to \mathcal{L}(p,q; a)$. Thus there is a genus one Heegaard splitting $V_1 \cup V_2$ of $|\mathcal{L}(p,q; a)|$ such that $\bar K$ is a Berge-Gabai knot of winding number $n$ in $V_1$ and $\Sigma(\mathcal{L}(p,q;a))$ is a closed submanifold of the core $C_2$ of $V_2$. As $|\mathcal{L}(p,q;a)| = L(p, q)$, the algebraic intersection number of a meridian curve of $V_1$ with one of $V_2$ is $\pm p$. By definition, $\gcd(p, n) = 1$, so Theorem \ref{th_fibred} implies that there is a locally trivial fibring of the exterior of $\bar K$  by surfaces which intersect $\partial V$ in curves parallel to the meridian of $V_2$. Therefore we can extend the fibration over the exterior of $K$ in $L(p,q) = |\mathcal{L}(p,q; a)|$ in such a way that it is everywhere transverse to $\Sigma(\mathcal{L}(p,q;a))$. Hence the fibration lifts to a fibring of the exterior of $K$. \end{proof}

\begin{proof}[Proof of Theorem \ref{th_fibred}]

Let $K$ be the closed $1\mbox{-bridge}$ braid contained in the interior of a solid torus $V$ determined by the three parameters: 
\begin{itemize}
\vspace{-.1cm} \item $n$, the braid index of $K$; 

\item $b$, the bridge index of $K$;

\item $t$, the twisting number of $K$.
\end{itemize}
\vspace{-.1cm} See \cite{Gabai3} for an explanation of these parameters and Figure \ref{fig_nodi} for an example. (Our conventions differ from those of \cite{Gabai3} by mirroring and changing orientation. This modification is convenient for presenting the knot's fundamental group.)

\begin{figure}[!ht]
\centerline{\includegraphics{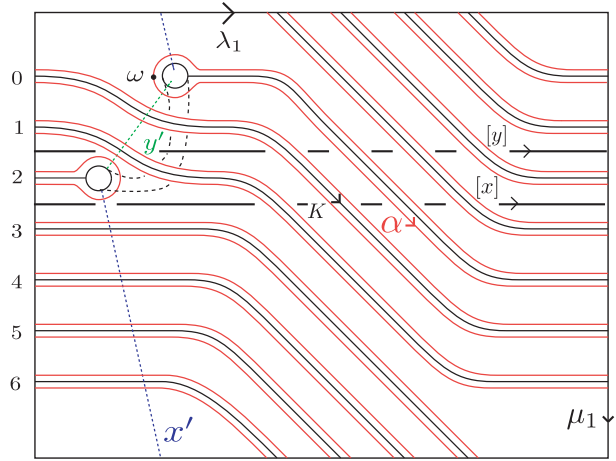}} 
\caption{{\footnotesize The Fintushel-Stern knot ($n=7,$ $b=2,$ $t=4$). 
The curve $x'$ is obtained from the arc labeled $x'$ by closing it in the boundary of the tunnel with an arc parallel to the bridge and $y'$ is obtained similarly by closing the arc $y'$ in the boundary of the tunnel. Here $R$ is: $y\ x\ y\ x\  x\ y\  x\ x\ y^{-1}x^{-1}y^{-1}x^{-1}x^{-1}y^{-1}x^{-1}x^{-1}$.}}

\label{fig_nodi} 
\end{figure}

Number the braid's strands successively $\overline 0$ to $\overline {n-1}$ and let $\sigma_i$ denote the $i^{th}$ elementary braid in which the $i^{th}$ strand passes over the $(i+1)^{st}$. The braid associated to $K$ has the following form: $\beta(K)=\sigma_{b-1}\cdots \sigma_0\delta^t$
where $\delta=\sigma_{n-2}\cdots\sigma_0$ is the positive $2\pi/n$ twist. Denote by $\pi$ the permutation of
$\mathbb{Z}/n$ determined by $\beta(K)$. It has the following simple form:
\begin{equation}
  \pi(\overline a)=\left\{
  \begin{array}{ll}
    \overline {a+t+1} & \mbox{if $0\leq a <b$}\\
    \overline t & \mbox{if $a=b$}\\
    \overline {a+t}& \mbox{if $b < a <n$}\\
  \end{array}
  \right.
  \label{eq:permutarea}
\end{equation}
for some $a \in \bar a$. As $K$ is a knot, $\pi$ is an $n$-cycle. 

Let $T_1=\partial V$ and $T_2 = \partial N(K)$ the boundary of a closed tubular neighborhood of $K$ in $\hbox{int}(V)$. There is a meridian class $\mu_1 \in H_1(T_1)$ well-defined up to $\pm 1$ and represented by the boundary of a meridian disk of $V_1$. Let $\lambda_1 \in H_1(T_1)$ be any class which forms a basis of $H_1(T_1)$ with $\mu_1$. Then $\lambda_1$ generates $H_1(V)$.   

Let $M$ denote the exterior of $K$ in $V$ and fix an essential simple closed curve $C$ on $\partial V$. We are clearly done if $C$ is a meridian curve of $V$, so assume that this is not the case. Then we can orient $C$ and find coprime integers $p \geq 1, q$ so that
$$[C] = q \mu_1 + p \lambda_1 \in H_1(T_1)$$ 
Note that $p$ is the algebraic winding number of $C$ in $V$. Assuming that $\gcd(p, n) = 1$ we must show that there is a locally trivial fibring of $M$ by surfaces which intersect $\partial V$ in curves parallel to $C$. The tools we use to prove this are Brown's theorem \cite{Br} and Stallings' fibration criterion \cite{Sta2}. See also \cite{OzSzLS} where a similar argument is invoked; our proof is only slightly more involved.
Brown's theorem gives necessary and sufficient conditions under which a homomorphism from a two-generator one-relator group
to $\mathbb{Z}$ has finitely generated kernel and Stallings' theorem produces a fibration of a $3$-manifold given such a homomorphism of its fundamental group.  More precisely:

\begin{theorem} {\rm (Theorem 4.3 and Proposition 3.1 of \cite{Br})} \label{Brown's Theorem} 
Let $G = \langle x,y : R \rangle$ be a two-generator one-relator group with $R = R_1R_2 \ldots R_m$,  $R_i \in \{x, x^{-1}, y, y^{-1}\}$, a cyclically
reduced and non-trivial relator. Let $S_1,\dots S_m$ be the proper initial segments of the relator $R$, i.e. $S_i=R_1\dots R_{i-1}$. 
Finally let $\varphi:G\to \mathbb{R}$ be a non-zero homomorphism. If $\varphi(x) \neq 0$ and $\varphi(y) \neq 0,$ then $\ker(\varphi)$ is finitely 
generated if and only if the sequence $\{\varphi(S_i)\}_{i=1}^m$ assumes its maximum and minimum values exactly once.
\end{theorem}

It is easy to see that the exterior $M$ of $K$ is homeomorphic to a genus $2$ handlebody with a $2$-handle attached to it. Start with a solid torus $U' \subset \hbox{int}(V)$ obtained by removing a small open collar of $T_1$ in $V$. Denote $\partial U'$ by $T_3$. As $K$ is $1$-bridge, it can be isotoped into $U'$ so that the bridge is a properly embedded arc and its complement, $\gamma$ say, is contained in $T_3$. Fix a disk neighborhood $D \subset T_3$ of $\gamma$ and let $\alpha = \partial D$. Let $U$ be the exterior  of the bridge in $U'$, a genus two handlebody. We can assume that $T_3 \setminus \partial U \subset \hbox{int}(D)$ and therefore $\alpha \subset \partial U$. By construction, $\alpha$ bounds a $2$-disk properly embedded in $\overline{V \setminus U}$ (i.e. a copy of $D$ isotoped rel $\partial D$ into $\overline{V \setminus U}$). It is easy to see that $M$ is a regular neighborhood of the union of $U$ and this disk.  

The fundamental group of $U$ is free on two generators $x, y$ represented by two curves in $T_3$ representing $\lambda_1$. (See Figure \ref{fig_nodi}.)  There are a pair of dual curves $x', y' \subset \partial U$ to these generators. This means that 
\begin{itemize}
\vspace{-.1cm} \item $x'$ and $y'$ bound disks in $U$; 

\item $x$ intersects $x'$ transversely in one point and is disjoint from $y'$;

\item $y$ intersects $y'$ transversely in one point and is disjoint from $x'$.
\end{itemize}
\vspace{-.1cm}  See Figure \ref{fig_nodi}. The word $R \in \pi_1(U)$ in $x, y$ represented by the curve $\alpha$ can be read off in the usual way: each signed intersection of $\alpha$ with $x'$, resp. $y'$, contributes $x^{\pm 1}$, resp. $y^{\pm 1}$,  while traveling around $\alpha$. 

We introduce the auxiliary function
$f:\mathbb Z /n \setminus\{\bar b\} \to \{x,y\}$ given by:

\begin{equation}
  f(\bar a)=\left\{
  \begin{array}{ll}
    y & \mbox{if $0\leq a <b$}\\
    x & \mbox{if $b < a <n$}\\
  \end{array}
  \right.
  \label{eq:functia_f}
\end{equation}

for some $a \in \bar a$. Let $w_j = f(\pi^j(\bar b))$ and consider the word $w = w_1 w_2 \dots w_{n-1} $. Then $R = ywxy^{-1}w^{-1}x^{-1} $. To see this, start with $y$ from the base point $\omega$ (c.f. Figure \ref{fig_nodi}); then follow the knot until the $b$ strand, which contributes $w$; then turn at the lower foot of 
the handle, which contributes $xy^{-1}$; then walk along the knot in the opposite direction until the strand $b$ is reached, which contributes  $w^{-1}$;  then close by passing $x'$, which contributes to the final $x^{-1}$. Notice that $R$ is cyclically reduced. It follows that
$$\pi_1(M) = \langle x,y : ywxy^{-1}w^{-1}x^{-1} \rangle$$

Let $\mu_2 \in H_1(T_2)$ be a meridinal class of $K$. The reader will verify that we can choose the longitudinal class $\lambda_1$ for $V$, a longitudinal class $\lambda_2 \in H_1(T_2)$ for $K$, and possibly replace $\mu_1$ by $-\mu_1$ so that in $H_1(M)$: 
\begin{itemize}
\vspace{-.1cm} \item $n \lambda_1 = \lambda_2$; 

\item $\mu_1 = n \mu_2$; 

\item $[yx^{-1}]=\mu_2$ (i.e. $[yx^{-1}]$ is represented by a meridian of $K$ at the bridge); 

\item $\lambda_1 + t \mu_2 = [x]$ (i.e. $\lambda_1$ and $[x]$ co-bound an annulus in $V$ which $K$ punctures $t$ times).
\end{itemize}
\vspace{-.1cm}
Consider the homomorphism $\pi_1(U) \to \mathbb Z$ which sends $x$ to $pt-nq \ne 0$ and $y$ to $pt-nq+p \ne 0$. Since the exponent sum of both $x$ and $y$ in $R$ is zero, it induces a homomorphism $\varphi: \pi_1(M) \to \mathbb Z$. Since $\gcd(p, nq) = 1$, $\varphi$ is surjective. From the above, it can then be verified that $\varphi(\lambda_1) = -nq$ and $\varphi(\mu_1) = np$. Hence $\varphi(\mu_1^q \lambda_1^p)=0$. 
 
\begin{lemma}\label{lm_fg} Let $S_1, S_2, \ldots, S_{2n+2}$ be the proper initial segments of $R = ywxy^{-1}w^{-1}x^{-1} = R_1R_2 \ldots R_{2n+2}$ where $R_i \in \{x, x^{-1}, y, y^{-1}\}$. Then the sequence $\{\varphi(S_i)\}_{i=1}^{2n+2}$ achieves its maximum and minimum values exactly once.
\end{lemma}

\begin{proof}

By construction, $\varphi(x) \neq 0$, $\varphi(y) \neq 0$, and $\varphi(y)>\varphi(x)$. The conclusion of the lemma is easily seen to hold when $\varphi(x)$ and $\varphi(y)$ have the same sign, so assume that $\varphi(x) < 0 < \varphi(y)$. 

Set $S = \max\{ \varphi(S_i): 1 \leq i \leq 2n+2 \}$ and $s = \min\{ \varphi(S_i): 1 \leq i \leq 2n+2 \}$. 

Since $\varphi(x) < 0 < \varphi(y)$ we have
\begin{equation}
\left\{
\begin{array}{l}
s \leq \varphi(S_{n+2}) < \varphi(S_{n+1}) < \varphi(S_n) \leq S \\ \\
s \leq \varphi(S_{n +i}) = \varphi(S_{n - i + 2}) + \varphi(x) - \varphi(y) < \varphi(S_{n - i + 2}) \leq S \hbox{ for } 3 \leq i \leq n + 1 \\ \\
s \leq \varphi(S_{2n+1}) = \varphi(S_{2n+2}) + \varphi(x) <  \varphi(S_{2n+2}) = 0 < \varphi(y) = \varphi(S_1) \leq S 
\end{array} \right.
\end{equation}
Thus the maxima of $\{\varphi(S_i)\}_{i=1}^{2n+2}$ can only occur in the sequence $\varphi(S_1), \varphi(S_2), \ldots , \varphi(S_n)$ and the minima in $\varphi(S_{n+2}), \varphi(S_{n+3}), \ldots , \varphi(S_{2n+1})$. 

We look at the maxima of $\{\varphi(S_i)\}_{i=1}^{2n+2}$ first. Suppose that $1 \leq l < r \leq n$. We claim that $\varphi(R_{l+1})+\dots + \varphi(R_r) \not \equiv 0 \hbox{ (mod $n$)}$. If so, $\varphi(S_l) \ne  \varphi(S_r)$ and therefore $S$ occurs precisely once amongst the values $\{\varphi(S_i)\}_{i=1}^{n}$.

Let $\overline{\varphi}$ be the reduction of $\varphi$ modulo $n$. Since $\gcd(p,n)=1$, we can define 
$$\hat \varphi = \overline p^{-1} \overline{\varphi}: \pi_1(M)\to \mathbb{Z}/n$$
Then $\hat\varphi(x) = \overline t$ and $\hat\varphi(y) = \overline{t+1}$ and therefore 
$$\hat\varphi(f(\overline a)) = \pi(\overline a)-\overline a$$ 
for all $a \in \mathbb{Z}/n \setminus\{\overline b\}$. Hence 
$\hat\varphi(R_{l+1}) + \dots + \hat\varphi(R_r) 
= \hat\varphi(w_l) + \dots + \hat\varphi(w_{r-1}) 
= \hat\varphi(f(\pi^l(\overline b))) + \dots + \hat\varphi(f(\pi^{r-1}(\overline b))) 
= (\pi^{l+1}(\overline b) - \pi^{l}(\overline b)) + \dots + (\pi^{r}(\overline b )- \pi^{r-1}(\overline b))
= \pi^{r}(\overline b) - \pi^{l}(\overline b)$. 
Since $\pi$ is an $n$-cycle and $1 \leq l < r \leq n$ we see that $\pi^r(\overline b) \not= \pi^l(\overline b)$. 
It follows that $\varphi(R_{l+1})+\dots + \varphi(R_r) \not \equiv 0 \hbox{ (mod $n$)}$.

The uniqueness of the minimum follows along the same lines. We saw above that the minima of $\{\varphi(S_i)\}_{i=1}^{2n+2}$ only occur in $\varphi(S_{n+2}), \varphi(S_{n+3}), \ldots , \varphi(S_{2n+1})$. As before, $\varphi(R_{l+1})+\dots + \varphi(R_r) \not \equiv 0 \hbox{ (mod $n$)}$ for all $n+2 \leq l < r \leq 2n + 1$ and therefore $\varphi(S_{n+2}), \varphi(S_{n+3}), \ldots , \varphi(S_{2n+1})$ are pairwise distinct. This implies the desired conclusion. 
\end{proof}

We can now complete the proof of Theorem \ref{th_fibred}. The previous lemma couples with Theorem \ref{Brown's Theorem} to show that the kernel of $\varphi$ is finitely generated. Stallings' fibration criterion \cite{Sta2} implies that $M$ admits a locally trivial surface fibration with fibre $F$ 
such that $\pi_1(F) = \ker(\varphi)$. Since $\varphi(\mu_1) = np \ne 0$ while $\varphi(\mu_1^q\lambda_1^p) = 0$, 
$\ker(\varphi |_{\pi_1(T_1)})$ is the infinite cyclic subgroup of $\pi_1(T_1)$ generated by $[C]$. Hence the fibration meets $T_1$ in curves parallel to $C$. To complete the proof, we must show that the intersection of a fibre $F$ with $T_1$ has $n$ components. 

To that end, note that as $\varphi$ is surjective we can orient $F$ so that for each $\zeta \in H_1(M)$ we have $\varphi(\zeta) = \zeta \cdot [F]$. Let $\phi_1 \in H_1(M)$ be the class represented by the cycle $F \cap T_1$ with the induced orientation. Clearly, $\phi_1 = \pm |F \cap T_1| [C]$. Since $\varphi(\lambda_1) = -nq$ and $\varphi(\mu_1) = np$, $\varphi(\pi_1(T_1)) = n \mathbb Z$. Thus if $\zeta \in H_1(M)$ is represented by a dual cycle to $[C]$ on $T_1$, then
$$n = \varphi(\zeta) = \zeta \cdot [F] = |\zeta \cdot \phi_1| = ||F \cap T_1| \zeta \cdot [C]| = |F \cap T_1|$$
This completes the proof.  
\end{proof}

\begin{proof}[Proof of Proposition \ref{fibred faces}] 
Let $K$ be a hyperbolic $1$-bridge braid on $n$ strands in a solid torus $V$. We use the notation developed in the proof of Theorem \ref{th_fibred}. In particular, $M$ is the exterior of $K$ in $V$ and $H_1(M) \cong \mathbb Z \oplus \mathbb Z$ with basis $\lambda_1, \mu_2$. By construction there are classes $\xi_1, \xi_2 \in H_2(M, \partial M)$ such that if $\partial: H_2(M, \partial M) \to H_1(\partial M)$ is the connecting homomorphism, then $\partial \xi_1 = \mu_1 - n \mu_2$ and $\partial \xi_2 = n \lambda_1 - \lambda_2$. Since $|\lambda_1 \cdot \xi_j| = \delta_{1j}$ and $|\mu_2 \cdot \xi_j| = \delta_{2j}$, $\{\xi_1, \xi_2\}$ is a basis for $H_2(M, \partial M) \cong H^1(M) \cong \mathbb Z \oplus \mathbb Z$. 

Consider the homomorphism $\psi$ given by the composition $H_2(M, \partial M) \stackrel{\partial}{\longrightarrow} H_1(\partial M) = H_1(T_1) \oplus H_1(T_2) \to H_1(T_1)$. Then $\psi(a \xi_1 + b \xi_2) = a \mu_1 + nb \lambda_1$, and therefore $\psi$ is injective.

Let $p, q$ be coprime integers such that $\gcd(n, p) = 1$. According to Theorem \ref{th_fibred}, there is a fibre $F$ in $M$ which can be oriented so that $\psi([F]) = [F \cap T_1] = nq \mu_1 + np \lambda_1 = \psi(nq \xi_1 + p \xi_2)$. Hence $[F] = nq \xi_1 + p \xi_2$ so that $nq \xi_1 + p \xi_2$ is a fibre class in $H_2(M, \partial M)$. 

Fix coprime integers $a, b$ and consider the class $\xi = a \xi_1 + b \xi_2$. The proposition will follow if we can show that the projective class of $\xi$ can be arbitrarily closely approximated by fibre classes \cite[Theorem 2]{Th}. By the previous paragraph $\xi$ is a fibre class when $a = 0$, so suppose this is not the case. It suffices to show that $\frac{b}{a} = \lim_m \frac{b_m}{a_m}$ where $a_m \xi_1 + b_m \xi_2$ are fibre classes. This is easy to verify: for each integer $m > 0$ set $p_m = nmba + 1$ and $q_m = mb^2$. Then $\gcd(p_m, nq_m) = 1$ and from the previous paragraph we see that $n q_m \xi_1 + p_m \xi_2$ is a fibre class. Finally, $\lim_m \frac{nq_m}{p_m} = \frac{b}{a}$, which completes the proof.
\end{proof}

\subsection{$K$ is not periodic} 

In this case, $Z(K)$ is generated by a free symmetry of the pair $(S^3,K)$. Then $\mathcal{Z}_K$ is a complete  hyperbolic 3-manifold with a torus cusp and $\mathcal{Z}_{K}(r(K)) = \mathcal{L}_{K}$ is a lens space $L(p, q)$. The image $\bar K$ of $K$ in $\mathcal{Z}_{K}(r(K))$ is primitive, since its preimage in the universal cover $S^3$ has one component. Since $|\mathcal{CC}(K)| > 1$, Proposition \ref{prop:commensurable} shows that there is another slope $r'$ in the torus cusp of $\mathcal{Z}_{K}$ such that 
$\mathcal{Z}_{K}(r')$ is a lens space $L(p', q')$. The following key result has been explained to us by Jake Rasmussen. 

\begin{theorem} \label{thm:fibred}  
Let $K$ be a primitive knot in $Y = L(p,q)$ which admits a non-trivial lens space surgery. Then $K$ is fibred.
\end{theorem}

This theorem shows that $\mathcal{Z}_{K}$ is a surface bundle over the circle, and since there is an unbranched cover $S^3 \setminus K \to \mathcal{Z}_K$, $K$ is a fibred knot. Many, though not all, of the elements of the proof of Theorem \ref{thm:fibred} are contained in \cite{Ras}. Owing to its importance to this paper, we include a proof here.

\begin{proof}[Proof of Theorem \ref{thm:fibred}] The analogous result is known to be true for knots in the 3-sphere \cite{NiF}: If a knot $K\subset S^3$ has a lens space surgery, then $K$ is fibred.  The proof of Ni's result uses the Heegaard-Floer homology package developed by Ozsv{\'a}th and Szab{\'o} in \cite{OzSzHD,OzSzKI} and extended 
to sutured manifolds by Ni and Juhasz \cite{NiF,JuhHD}. The essential property of lens spaces which is invoked is that they are L-spaces, which are rational homology spheres with Heegaard-Floer homology as simple as possible  (c.f. \cite{OzSzLS}). Our situation is similar in that both the initial and the surgered manifold are L-spaces, and our proof follows that of the $S^3$ case.

Let $\mu_K$ denote the meridinal slope of $K$ and $\lambda \ne \mu_K$ a slope whose associated surgery yields a lens space. By the cyclic surgery theorem, $\Delta(\lambda, \mu_K) = 1$, so any representative curve for $\lambda$ runs parallel to $K$.  

We find it convenient to use the notation from \cite{OzSzQS} even though it is somewhat different from that used elsewhere in the paper. We review this notation here. 

We use $\overline K$ to denote the knot $K$ with a choice of orientation. Dehn surgery on $K$ with slope $\lambda$ will be written $Y_\lambda(K)$. In \cite{OzSzQS}, Ozsv{\'a}th and Szab{\'o} compute the Heegaard-Floer homology of manifolds obtained by surgery on knots in rational homology spheres in terms
of the knot filtration on the chain complex whose homology is the Heegaard-Floer homology of the ambient manifold.
Based on this, Rasmussen computes the knot Floer homology of knots in lens spaces which admit {\it integer}  homology L-space surgeries \cite[Lemma 4.7]{Ras}.
The strategy here is to do the same calculation for knots admitting general L-space surgeries and then to pass to the Floer homology of a certain sutured manifold.

One can associate a doubly pointed Heegaard diagram $(\Sigma,\alphab,\betab,w,z)$ which determines $(Y,\overline K),$ from which Ozsv{\'a}th and Szab{\'o} construct a chain complex $CFK^{\infty}(\Sigma,\alphab,\betab,w,z)$ as follows. 
The generating set is $\{[\x,i,j]: \x\in \mathbb{T}_\alpha\cap\mathbb{T}_\beta,i,j\in\mathbb{Z}\}$ where $\mathbb{T}_\alpha$ and $\mathbb{T}_\beta$ are two totally real tori in the symmetric product $Sym^g(\Sigma)$ which is endowed with an almost complex structure. 
The differential counts certain pseudoholomorphic disks connecting the generators with the boundary mapping to $\mathbb{T}_\alpha\cup\mathbb{T}_\beta$. 
The two basepoints $w$ and $z$ give rise to codimension 2 submanifolds $\{w\}\times Sym^{g-1}(\Sigma),$ resp. $\{z\}\times Sym^{g-1}(\Sigma).$ The indices $i,j$ are
employed in order to keep track of the intersection of the holomorphic disks with the two submanifolds above. More precisely:
\begin{equation*}
  \partial^\infty[\x,i,j]=\sum_{\y\in\mathbb{T}_\alpha\cap\mathbb{T}_\beta}\sum_{\phi\in\pi_2(\x,\y),\mu(\phi)=1}\#\left( \frac{\mathit{M}(\phi)}{\mathbb{R}} \right)\cdot[\y,i-n_w(\phi),j-n_z(\phi)],
\end{equation*}

where $\pi_2(\x,\y)$ denotes the homotopy class of disks connecting $\x$ and $\y,$ $\mu(\phi)$ is the Maslov index of $\phi,$ $\#\left( \frac{\mathit{M}(\phi)}{\mathbb{R}} \right)$ is the count of
holomorphic representatives of $\phi,$ $n_w(\phi)=\#\phi\cap\left( \{w\}\times Sym^{g-1}(\Sigma) \right)$, and similarly for $n_z(\phi)$. Note that $n_w(\phi)\ge 0,n_z(\phi)\ge 0$ since the
submanifolds involved are almost complex. Therefore the $i,j$ indices define a $\zz\oplus\zz$ filtration on $CFK^\infty(Y,\overline K).$

The non-triviality of $\phi\in\pi_2(\x,\y)$ is homologically obstructed and as consequence the complex $CFK^\infty(Y, \bar K)$ splits into summands which are in 
bijection with $\Spinc$ structures on $Y.$ Following Turaev \cite{Tur} $\Spinc$ structures can be seen as homology classes of non-vanishing vector fields and they form an affine space over 
$H^2(Y).$ From the combinatorics of the Heegaard diagram one can construct a function $\s_w:\mathbb{T}_\alpha\cap\mathbb{T}_\beta\to \Spinc(Y)$ which 
sends an intersection point $\x$ to the homology class of a vector field. There is also a relative version, $\underline{\Spinc}(Y,K),$ which consists of homology classes
of vector fields on $Y\setminus N(K)$ which point outwards at the boundary; one has an analogous map $\s_{w,z}:\mathbb{T}_\alpha\cap\mathbb{T}_\beta\to\RSpinc(Y,K)$ \cite[Section 2.4]{OzSzQS}. $\RSpinc(Y,K)$ is an affine space over $H^2(Y,K).$ 

One can extend a vector
field on $Y\setminus N(K)$ to a vector field on $Y$ such that the (oriented) knot is a closed trajectory. This gives rise to a map $G_{Y,\overline K}:\RSpinc(Y,K)\to\Spinc(Y)$ which
is equivariant with respect to the action of $H^2(Y,K):$ $G_{Y,\overline K}(\xi+k)=G_{Y,\overline K}(\xi)+\iota^*(k).$ where $\iota^*:H^2(Y,K)\to H^2(Y)$ is the natural map induced by inclusion.
Moreover, given $\xi,\eta\in\RSpinc(Y,K),$ $G_{Y,\overline K}(\xi)=G_{Y,\overline K}(\eta)\iff \xi=\eta+n\cdot PD[\mu]$ for some $n\in\zz$ where $\mu$ is a meridian of
$K.$ There is an analogous map $G_{Y_\lambda(K),\overline K}:\RSpinc(Y,K)\to\Spinc(Y_\lambda(K))$ which extends the vector fields on $Y\setminus N(K)$ to $Y_\lambda(K)$ 
such that the induced knot $K'$ in the surgered manifold becomes a trajectory with the orientation inherited from $K$. See \cite[Section 6]{OzSzQS}.

For $\x,\y\in \mathbb{T}_\alpha\cap\mathbb{T}_\beta$ such that there exists $\phi\in\pi_2(\x,\y)$, we have: $\s_{w,z}(\x)-\s_{w,z}(\y)=(n_z(\phi)-n_w(\phi))\cdot PD[\mu]$ \cite[Lemma 2.1]{OzSzQS}.
This splits $CFK^\infty(Y,\overline K)$ into various summands: Fix $\xi\in\RSpinc(Y,K).$ The subgroup $C_\xi:=\{[\x,i,j]\in CFK^\infty(Y,\overline K): \s_{w,z}(\x)+(i-j)\cdot PD[\mu]=\xi\}$ becomes
a subcomplex of $CFK^\infty(Y,\overline K).$

Ozsv{\'a}th and Szab{\'o} consider the induced complexes $A_\xi(Y,\overline K)$ and $B_\xi(Y,\overline K)$ as ingredients in the Morse surgery formula:
The complex $A_\xi(Y,\overline K):=C_\xi\{\max(i,j)=0 \}$ with the induced differential from the complex $CFK^\infty(Y,\overline K)$ computes the
Heegaard Floer homology of large enough integral surgeries on $K$ in a particular $\Spinc$ structure \cite[Theorem 4.1]{OzSzQS}. Since $\chi(\HF(M,\s))=1$ 
for all $\s\in \Spinc(M)$ where $M$ is any rational homology sphere \cite[Theorem 5.1]{OzSzPA} and framed surgeries on $K$ are rational homology spheres, 
we have that $\hbox{rank}(H_*(A_\xi(Y,\overline K)))\ge 1.$

By definition, $B_\xi(Y,\overline K):=C_\xi\left\{ i=0 \right\}$ computes $\widehat{HF}(Y,G_{Y,\overline K}(\xi)).$ 
In addition, $C_\xi\{j=0\}$ is identified with $B_{\xi+PD[K_\lambda]},$ where $K_\lambda$ is the knot $K$ pushed off inside $Y\setminus N(K)$ along the framing $\lambda$. See \cite[Proposition 3.2]{OzSzQS} for an explanation of the grading shift.

Consider also the maps $v_\xi$, resp. $h_\xi$, the natural projections $C_\xi\{\max(i,j)=0\}\to C_\xi \left\{ i=0 \right\}$,  resp. $C_\xi\{\max(i,j)=0\}\to C_\xi \left\{ j=0 \right\}.$

By \cite[Theorem 4.1]{OzSzQS} they are identified with the induced maps in $\HF$ by the natural cobordism $W'_m(K)$ between $Y_\lambda(K)$ and $Y$ equipped with corresponding $\Spinc$ structures.

We can now state the surgery formula for the ``hat" version of Heegaard-Floer homology.  This corresponds to taking $\delta=0$ in \cite[Theorem 6.4]{OzSzQS}. See \cite[Section 2.8]{OzSzQS} for an explanation.
We therefore drop the $\delta$ indice in what follows:

\begin{theorem} \label{th:mapp_cone}
{\rm \cite[Theorem 6.4]{OzSzQS}} Fix a $\Spinc$ structure $\s\in \Spinc(Y_\lambda(K)).$ Then 
  \begin{equation}
\widehat{HF}(Y_\lambda(K),\s) \cong H_*(\hbox{Cone}(D_\mathfrak{s}:\mathbb{A}_\mathfrak{s}(Y,K)\to\ \mathbb{B}_\mathfrak{s}(Y,K)))  
\end{equation}

where 

\begin{equation*}
  \mathbb{A}_\mathfrak{s}(Y,K)= \bigoplus_{\{\xi\in \RSpinc(Y_\lambda(K),K)| G_{Y_\lambda(K),\overline K}(\xi)=\mathfrak s\}} A_\xi(Y,\overline K)
 \end{equation*}

and 

\begin{equation*}
  \mathbb{B}_\mathfrak{s}(Y,K)= \bigoplus_{\{\xi\in \RSpinc(Y_\lambda(K),K)| G_{Y_\lambda(K),\overline K}(\xi)=\mathfrak s\}} B_\xi(Y,\overline K)
\end{equation*}

The map $D_\s$ is defined as follows:

\begin{equation*}
  D_\s\left(\{a_\xi \}_{\xi\in G^{-1}_{Y_\lambda(K),\overline K}(\s)}\right)=\{b_\xi\}_{\xi\in G^{-1}_{Y_\lambda(K),\overline K}(\s)}
\end{equation*}

with $b_\xi=h_{\xi-PD[K_\lambda]}(a_{\xi-PD[K_\lambda]})-v_\xi(a_\xi).$

\end{theorem}

In our case, both $Y$ and $Y_\lambda(K)$ are lens spaces, hence L-spaces \cite{OzSzLS}. Thus $H_*(B_\xi(Y,\overline K))\cong \mathbb{Z}$ for any $\xi\in\Spinc(Y,K)$ and
$H_*(\hbox{Cone}(D_\s))\cong\mathbb{Z}$ for any $\s\in \Spinc(Y_\lambda(K))$  by Theorem \ref{th:mapp_cone}. In fact for any field $\mathbb{F},$ $H_*(B_\xi(Y,\overline K);\mathbb{F})\cong \mathbb{F}$ 
and $H_*(\hbox{Cone}(D_\s);\mathbb{F})\cong\mathbb{F}.$

\begin{lemma} After a possible change of orientation of the ambient manifold $Y,$ $H_*(A_\xi(Y,\overline K))\cong\mathbb{Z}$ for any $\xi\in\RSpinc(Y,K).$ 
\end{lemma}

\begin{proof}

This lemma is a slight generalisation of \cite[Lemma 4.6]{Ras}. It is only applied to a summand corresponding to the particular $\Spinc$ structure $\s$. Write the rational longitude of $K$ as $a\cdot \mu+ p\cdot \lambda$ for some $a\in\zz\setminus\left\{ 0 \right\}.$ By changing the orientation of $Y$ if necessary 
one can assume $a<0.$

The proof of Proposition \ref{prop:commensurable} shows that $K' \subset Y_\lambda(K)$ is primitive in $Y_\lambda(K)$. The map $G_{Y_\lambda(K),\overline K}$ is affinely modeled on the canonical projection: $\pi:\mathbb{Z}\to \mathbb{Z}/m$ where $m$ is the order of $H^2(Y_\lambda(K))$ and $\pi$ is the map $i^*$ induced in cohomology $i^*:H^2(Y_\lambda(K),K)\to H^2(Y_\lambda(K)),$ see \cite[Section 2.2]{OzSzQS}.
Therefore the groups $(A_\xi)_{\xi\in G^{-1}_{Y_\lambda(K),\overline K}(\s)}$ form an affine copy of $m\cdot\mathbb{Z}$ in $\mathbb{Z}\cong\Spinc(Y,K)$ and adding $PD[K_\lambda]$ corresponds to
a translation by $m.$    
By \cite[Lemma 6.5]{OzSzQS} and the assumption $a<0$, for sufficiently large $n>n_+$ the map $v_{\xi+n\cdot PD[K_\lambda]}:A_{\xi+n\cdot PD[K_\lambda]}(Y,\overline K)\to B_{\xi+n\cdot PD[K_\lambda]}(Y,\overline K)$ is an isomorphism and 
$h_{\xi+n\cdot PD[K_\lambda]}:A_{\xi+n\cdot PD[K_\lambda]}(Y,\overline K)\to B_{\xi+(n+1)\cdot PD[K_\lambda]}(Y,\overline K)$ is trivial. If $n$ is sufficiently small, $n<n_-$, $v_{\xi+n\cdot PD[K_\lambda]}$ is trivial and $h_{\xi+n\cdot PD[K_\lambda]}$ is an isomorphism.

In general, the homology of the mapping cone of $D_\s$ is an extension of $\hbox{Ker}( (D_\s)_*)$ by $\hbox{Coker}( (D_\s)_*)$ \cite[Chapter 1]{Weib}.
Using homology with field ($\mathbb{F})$ coefficients, this extension splits: 
$h_*(\hbox{Cone}(D_\s))\cong \hbox{Ker}((D_\s)_*)\oplus \hbox{Coker}((D_\s)_*).$ Another way to say this is:
$H_*(\hbox{Cone}(D_\s))\cong H_*(\mathbb{X})$ where $\mathbb{X}$ is the short chain complex: 

\begin{diagram}
  0  &\rTo & H_*(\mathbb{A}_\s(Y,\overline K);\mathbb{F}) & \rTo^{(D_\s)_*} & H_*(\mathbb{B}_\s(Y,\overline K);\mathbb{F}) & \rTo&0
\end{diagram}
\cite[Theorem 4.1]{Ras} 

Owing to the behavior of $D_\s|A_{\xi+n\cdot PD[K_\lambda]}(Y,\overline K)$ for large, resp. small, $n$, the chain complex $\mathbb{X}$ splits into an infinite sum of acyclic subcomplexes:  
  
\begin{diagram}
  0  &\rTo & H_*(A_{\xi+n\cdot PD[K_\lambda]}(Y,\overline K);\mathbb{F}) & \rTo^{(h_{\xi+n\cdot PD[K_\lambda]})_*}_\cong & H_*(B_{\xi+n\cdot PD[K_\lambda]}(Y,\overline K);\mathbb{F}) & \rTo&0
\end{diagram}
 for $n>n_+$,   
\begin{diagram}
  0  &\rTo & H_*(A_{\xi+n\cdot PD[K_\lambda]}(Y,\overline K);\mathbb{F}) & \rTo^{(v_{\xi+n\cdot PD[K_\lambda]})_*}_\cong & H_*(B_{\xi+(n+1)\cdot PD[K_\lambda]}(Y,\overline K);\mathbb{F}) & \rTo&0
\end{diagram}
for $n<n_-$,  
and the nontrivial subcomplex between the groups $A_{\xi+n_-\cdot PD[K_\lambda]}(Y,\overline K)$ and $A_{\xi+n_+\cdot PD[K_\lambda]}(Y,\overline K)$:

  \begin{diagram}
    A_{\xi+n_-\cdot K_\lambda}&     &  A_{\xi+(n_-+1)\cdot K_\lambda}&     &     \dots&     & A_{\xi} &  & A_{\xi+K_\lambda}  &  & \dots & & A_{\xi+n_+\cdot K_\lambda}\\
    & \rdTo_{h_{\xi+n_-\cdot K_\lambda}} & \dTo^{v} & \rdTo_{h_{\xi(n_-+1)\cdot K_\lambda}}& & \rdTo&      \dTo^{v_\xi}     &           \rdTo_{h_{\xi }}& \dTo^{v}     &   \rdTo_{h_{\xi+\cdot K_\lambda}}   & & \rdTo& \dTo_{v_{\xi+n_+\cdot K_\lambda}} \\
    &  &       B_{\xi+(n_-+1)\cdot K_\lambda} &  &       \dots &  &    B_\xi &             & B_{\xi+K_\lambda}      &             & \dots & & B_{\xi+n_+\cdot K_\lambda}\\
  \end{diagram}

Since $H_*(\mathbb{X};\mathbb{F}) \cong H_*(B_{\xi+n\cdot PD[K_\lambda]}(Y,\overline K);\mathbb{F})\cong \mathbb{F}$, $\hbox{rank}( H_*(A_{\xi+n\cdot PD[K_\lambda]}(Y,\overline K);\mathbb{F}))\ge 1$ and the number of $A$ groups is one greater than that of $B$ groups,
we have $H_*(A_{\xi+n\cdot PD[K_\lambda]}(Y,\overline K);\mathbb{F})\cong\mathbb{F}$. Since $\mathbb{F}$ was arbitrary, the universal coefficient theorem implies that $H_*(A_\xi(Y,\overline K))\cong\mathbb{Z}$.
\end{proof}

This phenomenon was studied in \cite{OzSzLS} for knots in $S^3.$ The result is purely algebraic, so it extends to our situation. As in the previous lemma, the proof is the same. The only change is that it is applied to a summand in the knot filtration corresponding to a fixed $\Spinc$ structure on $Y$.

\begin{lemma} {\rm (\cite[Lemmas 3.1 and 3.2]{OzSzLS})} \label{0 or z}
Under the conditions above, $\HFK(Y,K,\xi)$ is either $\mathbb{Z}$ or $0$ for any $\xi\in\RSpinc(Y,K).$
\end{lemma}
\begin{proof}

Fix $\xi\in\RSpinc(Y,K).$ Lemmas 3.1 and 3.2 in \cite{OzSzLS} apply to a general $\zz\oplus\zz$ filtered chain complex $C.$ We take $C$ to be $C_\xi,$
notice that $C\left\{ \max(i,j)=0 \right\}$ corresponds to $A_\xi(Y,\overline K)$ and $C\left\{ \max(i,j-m)=0 \right\}$ corresponds to $A_{\xi+m\cdot PD[\mu]}(Y,\overline K);$
in particular the hypotheses of the two lemmas are satisfied.
One can therefore apply the argument in the proof of Theorem 1.2 in \cite{OzSzLS} and the conclusion follows.
\end{proof}

Juhasz defined an Ozsv{\'a}th-Szab{\'o}-type invariant (\cite{JuhHD}) called \emph{sutured Floer homology - SFH} for (balanced) sutured manifolds $(M,\gamma)$. (See also \cite{NiF}.) 
One can construct a balanced sutured manifold $Y(K)$ starting from a knot $K$ by removing $N(K)$ and considering as sutures two copies of the meridian with opposite orientations.
It is easy to see that $SFH(Y(K))\cong \HFK(Y,K)$ by a natural identification between the corresponding chain complexes. \cite[Proposition 9.2]{JuhHD}.

The invariant $SFH$ also decomposes into different summands corresponding to $\Spinc$ structures on $Y(K)$ which are in affine bijection with $H^2(Y(K),\partial Y(K)),$ hence in
bijection with $\RSpinc(Y,K).$ The isomorphism above preserves the splitting along relative $\Spinc$ structures.
The invariant $SFH$ proves to be very strong in detecting tautness and products:
\begin{theorem}\label{th:sut}{\rm (\cite{JuhSD})} Let $(M,\gamma)$ be an irreducible, balanced sutured manifold. Then $(M,\gamma)$ is taut if and only if $SFH(M,\gamma)\neq 0$ and 
  it is a product sutured manifold if and only if $SFH(M,\gamma)\cong\mathbb{Z}.$
\end{theorem}

The knot $K$ is rationally null-homologous and primitive in $Y$. Hence there is a surface $F$ properly embedded in $M =Y \setminus N(K)$, the exterior of $K$, whose boundary is the rational longitude of $K$. Without loss of generality we assume $F$ has minimal genus, $g$ say, among all such surfaces. One can cut open $M$ along $F$ and construct a sutured manifold $Y(F)$ whose suture is a parallel copy of $\partial F$. (See \cite{Gab} for the original definitions of sutured manifolds and sutured manifold decompositions.) Our knot $K$ will be fibred if and only if $Y(F)$ is a product sutured manifold.
One can compute $SFH(Y(F))$ in terms of the knot Floer homology of $K$ via the surface decomposition theorem of Juhasz:

\begin{theorem}\label{thm:SFHdecomp}{\rm (\cite[Theorem 1.3]{JuhSD})} Let $(M,\gamma)$ be a strongly balanced sutured manifold and $F$ a decomposing surface, and denote the manifold resulting from the decomposition by $(M(F),\gamma(F)).$ Then 
  \begin{equation*}
    SFH(M(F),\gamma(F))\cong \bigoplus_{\s\in Out(F)} SFH(M,\s)
  \end{equation*}
  where $Out(F)$ are the \emph{outer} $\Spinc$ structures on $(M,\gamma)$ with respect to $F,$ i.e. the homology classes of vector fields
in which one can find a representative which is never a negative multiple of the normal to $F$ with respect to some Riemannian metric on $M$
\cite[Definition 1.1]{JuhSD}.
\end{theorem}

The strongly balanced hypothesis is a technical condition trivially satisfied in our case. 
The condition $\s\in Out(F)$ can be rephrased in terms of the Chern class of $\s$ evaluated on $F:$ $\s\in Out(F) \iff$ $<c_1(\s,t_0),[F]>=c(F,t_0),$ where $c(F,t_0)$ is a combinatorial quantity which in our case turns out to be: $c(F,t_0)=1-2g-p$ \cite[Section 3]{JuhSD}. See below for an explication of the terms in the formula:

The Chern class of a relative $\Spinc$ structure $\s$ is defined in the following way. 
Take a representative $v$ of $\s$ (i.e. a nowhere vanishing vector field on $M$ with predetermined behavior on $\partial M$ with respect to the sutures \cite[\S 3]{JuhSD}. Put a Riemannian metric on $M$ and consider 
the orientable 2-plane field $v^\perp$. Consider a trivialisation of $v^\perp_{\partial M}$ which exists because of the strongly balanced hypothesis. Then the Chern class of $\s$ relative to this trivialisation is the obstruction to 
extending the trivialisation to all of $M.$ (See \cite{JuhSD} for details in the sutured case and \cite{OzSzAp} for the knot complement case.) There is a natural trivialisation $t_0$ on $\partial M$ to consider, namely the section consisting of vectors parallel to the meridian of $K$. 

Since $H_1(M)$ contains no 2-torsion in our situation, the relative $\Spinc$ structures on $M$ are identified by their Chern class (see \cite{Gom} for the 
closed case; the relative case can be deduced by filling and applying the closed case result), which in 
turn are identified by the evaluation on the homology class $[F]$.
Hence $Out(F)$ consists precisely of one $\Spinc$ structure $\mathfrak{\xi_0}\in\RSpinc(Y,K)$ (see the next paragraph for the exact identification of $\mathfrak{\xi_0}$). Therefore, by Theorem \ref{thm:SFHdecomp}, $SFH(Y(F))\cong \HFK(Y,K,\mathfrak{\xi_0})$ which is 
$0$ or $\mathbb{Z}$ by Lemma \ref{0 or z}. As $K$ is primitive and $Y$ prime, $M$ is irreducible, and as $F$ is genus minimising, $Y(F)$ is taut. Thus $SFH(Y(F))\not\cong 0$ by Theorem \ref{th:sut} and so must be isomorphic to $\zz$. Hence $K$ is fibred.

In fact, as for knots in $S^3$, one can identify $SFH(Y(F))$ with the top summand with respect to the Alexander grading $\HFK(Y,K,g+(p-1)/2)$ (\cite[Section 3.7]{Ras}): In \cite{Ras} the Alexander grading $A$
on relative $\Spinc$ structures is defined such that the Euler characteristic of the Floer homology is symmetric with respect to the origin. The same grading (after the identification $H^2(Y,K)\cong \zz$ given by declaring $[F]$ to be the positive generator) is defined in
\cite[Section 4.4]{NiTN} in terms of the Chern class of the $\Spinc$ structures. By Juhasz's decomposition formula \cite[Lemma 3.10 and Theorem 1.3]{JuhSD}, we get $<c_1(\xi_0,t_0),[F]>=1-2g-p,$ hence $A(\xi_0)=(1-2g-p)/2$
and by conjugation invariance \cite[Section 4.4 Equation 2]{NiTN}, $SFH(Y(F))\cong \HFK(Y,K,g+(p-1)/2).$ 
\end{proof}

\subsection{Proof of Theorem \ref{thm: Ni analogue}}  

We prove Theorem \ref{thm: Ni analogue} here; it is an analogue of Ni's fibring theorem \cite{NiF} in an orbifold setting. Recall that we have assumed that $K$ is a knot in an orbi-lens-space $\mathcal{L}$ which is primitive in $\mathcal{L}$ and which admits a non-trivial orbi-lens space surgery.

\begin{proof}[Proof of Theorem \ref{thm: Ni analogue}]
When $\mathcal{L}$ is a manifold, this is just Theorem \ref{thm:fibred}. Suppose then that $\mathcal{L}$ has a non-empty singular set, say $\mathcal{L} = \mathcal{L}(p,q; a, b)$. Set $L_0 = \mathcal{L}(p,q; a, b) \setminus N(\Sigma(\mathcal{L}(p,q; a, b)))$ and note that as in the proof of Lemma \ref{lem:intersection}, 
$$L_0 \cong \left\{ 
\begin{array}{ll}
S^1 \times D^2 &  \hbox{ if } \vert \Sigma(\mathcal{L}(p,q; a, b)) \vert = 1 \\
S^1 \times S^1 \times [0,1] & \hbox{ if } \vert \Sigma(\mathcal{L}(p,q; a, b)) \vert = 2
\end{array} \right. 
$$
Since $K$ admits a non-trivial orbi-lens space surgery in $\mathcal{L}$, $L_0$ admits a non-trivial cosmetic surgery (c.f. the proof of Lemma \ref{lem:intersection}). Lemma \ref{Berge-Gabai} then shows that $L_0 \cong S^1 \times D^2$ (so we can suppose that $b = 1$) and $K$ is a Berge-Gabai knot in $L_0$ (Definition \ref{def:bg}). Let $n$ be the winding number of $K$ in $L_0$. Our hypotheses imply that $\gcd(p, n) = 1$. Thus Theorem \ref{th_fibred} implies that there is a locally trivial fibring of the exterior of $K$ in $L_0$ by surfaces which intersect $\partial L_0$ in curves parallel to the meridian slope of the solid torus $\overline{N(\Sigma(\mathcal{L}(p,q; a)))}$. Therefore we can extend the fibration over the exterior of $K$ in $\mathcal{L}(p,q; a)$ in such a way that it is everywhere transverse to $\Sigma(\mathcal{L}(p,q;a))$. We endow each fibre $F$ of this surface fibration with the structure of a $2$-orbifold by  declaring each point of $F \cap \Sigma(\mathcal{L}(p,q;a))$ to be a cone point of order $a$. In this way the exterior of $K$ in $\mathcal{L}(p,q; a)$ admits an orbifold fibring with base the circle.     
\end{proof} 

\subsection{Proof of Theorem \ref{th:properties}(1) and (2)} 
Let $K$ be a hyperbolic knot without hidden symmetries such that $|\mathcal{CC}(K)| > 1$. If $K$ is periodic, it is an unwrapped Berge-Gabai knot  (Proposition \ref{prop:berge-gabai}) and so Corollary \ref{cor:ubg periodic} implies that it is fibred. If $K$ is not periodic, then $Z(K)$ acts freely on $S^3$. Proposition \ref{prop:commensurable} shows that the image $\bar K$ of $K$ in the lens space $\mathcal{L}_K$ admits a non-trivial lens space surgery. Since $\bar K$ is primitive in $\mathcal{L}_K$, Theorem \ref{thm:fibred} shows that $\bar K$, and therefore $K$, is fibred. Thus Theorem \ref{th:properties}(1) holds. Part (2) of that theorem is an immediate consequence of the fibration result (1) and the fact that the knots are cyclically commensurable.
\qed

\section{Orientation reversing symmetries} \label{amphi}

In this section we prove assertion $(4)$ of Theorem \ref {th:properties}.

\begin{proposition} \label{prop:chiral}
Let $K$ be an amphichiral hyperbolic knot. Then $|\mathcal{CC}(K)| = 1$.  Moreover, if $K$ has no hidden symmetry, then  $|\mathcal{C}(K)| = 1$.
\end{proposition}


\begin{proof}
Let $K$ be an amphichiral knot with $S^3 \setminus K \cong {\bf H}^3/\Gamma_K$. Fix an orientation-reversing isometry $\theta: S^3 \setminus K \to S^3 \setminus K$ and lift it to $\tilde \theta \in \Isom({\bf H}^3)$.  Let $N(\Gamma_K)$ be the  normalizer of $\Gamma_K$ in $\Isom({\bf H}^3)$.

  
Then $\tilde \theta \in N(\Gamma_K)$ and  normalizes $N^+(\Gamma_K)$.  The action of $\tilde \theta$ permutes the index $2$ subgroups of $N^+(\Gamma_K)$ and so it leaves invariant the unique such subgroup with a torus cusp (c.f. Lemma \ref{onetorus}). Call this subgroup $\Gamma_Z$ and recall that ${\bf H}^3 /\Gamma_Z \cong  (S^3 \setminus K) / Z(K) \cong \mathcal{Z}_{K}$. Thus $\tilde  \theta$ induces an orientation-reversing isometry $\bar \theta: \mathcal{Z}_K \to \mathcal{Z}_K$ which lifts to $\theta$.

Let $\mu_K, \lambda_K$ be a meridian, longitude basis of the first homology of the cusp of $S^3 \setminus K$. It is clear that $\theta_*(\lambda_K) = \pm \lambda_K$ while $\theta_*(\mu_K) = \pm \mu_K$ by \cite{GL}. 
Projecting to $\mathcal{Z}_K$, we see that $\mu_K \mapsto \bar \mu$ and $\lambda_K \mapsto |Z(K)| \bar \lambda$ where $\bar \mu, \bar \lambda$ is a basis of the first homology of the cusp of $\mathcal{Z}_{K}$. Since $\bar \theta$ is orientation-reversing and lifts to $\theta$, there is an $\epsilon \in \{\pm 1\}$ such that 

$$\bar \theta_*(\bar \lambda) = \epsilon \bar \lambda \ \ \text{and} \ \ \bar \theta_*(\bar \mu) = - \epsilon \bar \mu.$$

\noindent It follows that, given any slope $\alpha = p \bar \mu + q \bar \lambda$ in the cusp of $\mathcal{Z}_K$, 
$$\Delta(\alpha, \bar \theta_*(\alpha)) = 2|pq| \equiv 0 \hbox{ (mod $2$)}.$$
Since the set of slopes in the cusp of $\mathcal{Z}_K$ whose associated fillings yield orbi-lens spaces is invariant under $\bar \theta$ and the distance between any two such slopes is at most $1$ (c.f. the proof of Lemma \ref{lem:intersection}), each such slope must be invariant under $\bar \theta$. But from the distance calculation immediately above, the only slopes invariant under $\bar \theta$ are those associated to $\bar \mu$ and $\bar \lambda$. The latter is the rational longitude of $|\mathcal{Z}_K|$ and so its associated filling cannot be an orbi-lens space. Thus the only slope which can yield an orbi-lens space is $r_K$, the slope associated to $\bar \mu$. Lemma \ref{lem:slope} then shows that there is exactly one knot complement in the cyclic commensurability class of $S^3 \setminus K$. 
\end{proof}

Proposition \ref{prop:chiral} together with Proposition \ref{prop:volume} directly implies part $(4)$ of Theorem \ref{th:properties}.

\begin{theorem} Let $S^3 \setminus K$ be a chiral knot complement without hidden symmetries.  Then $S^3 \setminus K$ is not commensurable with an orbifold  which admits an orientation-reserving involution.  That is, a knot complement without hidden symmetries in its orientable commensurator does not have hidden symmetries in its full commensurator. 
\end{theorem} 
\begin{proof} 

Suppose that $S^3 \setminus K$ is commensurable with an orbifold $\mathcal{O}$ which admits an orientation-reversing involution. Let $\Gamma_K$ and $\Gamma_\mathcal{O}$ be discrete subgroups of $\PSL(2,\mathbb{C})$ such that  ${\bf H}^3 / \Gamma_K \cong S^3 \setminus K$ and ${\bf H}^3 / \Gamma_{\mathcal{O}} \cong \mathcal{O}$.  We furthermore suppose that $\Gamma_\mathcal{O}$ and $\Gamma_K$ intersect in a finite-index subgroup, by conjugating if necessary.  By Mostow-Prasad rigidity, the  involution  of $\mathcal{O}$ corresponds to an element $g \in \Isom({\bf H}^3)$ which conjugates the fundamental group of $\mathcal{O}$ in $\PSL(2, \mathbb{C})$ to itself.  That is $g \Gamma_{\mathcal{O}} g^{-1} = \Gamma_{\mathcal{O}}$.  Thus $g$ is contained in the full commensurator of $\Gamma_\mathcal{O}$, which is the same as the full commensurator of $\Gamma_K$. This implies that $\Gamma_K$ is commensurable with $g\Gamma_kg^{-1}$, or that $S^3 \setminus K$ is commensurable with its image under an orientation-reversing involution.  But this  knot complement has the same volume, which contradicts Proposition \ref{prop:volume}. 

\end{proof}

\medskip\noindent

\vspace{.5cm} 
{\scriptsize

\noindent Michel Boileau, Institut de Math\'ematiques de Toulouse UMR 5219, Universit\'e Paul Sabatier, 31062 Toulouse Cedex 9,
et Institut Universitaire de France 103 bd Saint-Michel 75005 Paris, France 
\newline\noindent
e-mail: boileau@picard.ups-tlse.fr \\

\noindent
Steven Boyer, D\'ept. de math., UQAM, P. O. Box 8888, Centre-ville, Montr\'eal, Qc, H3C 3P8, Canada
\newline\noindent
e-mail: boyer@math.uqam.ca \\ 

\noindent
Radu Cebanu, D\'ept. de math., UQAM, P. O. Box 8888, Centre-ville, Montr\'eal, Qc, H3C 3P8, Canada
\newline\noindent
e-mail: radu.cebanu@gmail.com \\ 

\noindent
Genevieve S. Walsh, Dept. of Math., Tufts University, Medford, MA 02155, USA
\newline\noindent
e-mail: genevieve.walsh@gmail.com

}

\end{document}